\documentclass[10pt, reqno]{amsart}
\usepackage{tikz}
\usetikzlibrary{matrix,arrows,decorations.pathmorphing}
\usepackage{amssymb}

\newtheorem{proposition}{Proposition}[section]
\newtheorem{lemma}[proposition]{Lemma}
\newtheorem{corollary}[proposition]{Corollary}
\newtheorem{theorem}[proposition]{Theorem}
\newtheorem{remark}[proposition]{Remark}

\theoremstyle{definition}

\newcommand{\selabel}[1]{\label{se:#1}}

\newcommand{\eqlabel}[1]{\label{eq:#1}}
\newcommand{\equref}[1]{(\ref{eq:#1})}

\def\<{\leqslant}
\def\>{\geqslant}
\def\a{\alpha}
\def\b{\beta}

\def\g{\gamma}

\def\O{\Omega}

\def\ol{\overline}

\def\t{\triangle}
\def\e{\varepsilon}

\def\l{\lambda}

\def\ot{\otimes}
\def\ra{\rightarrow}

\date{}

\begin{document}
\title{Representations of Drinfeld Doubles of Radford Hopf algebras}
\author{Hua Sun}
\address{School of Mathematical Science, Yangzhou University,
	Yangzhou 225002, China}
\email{huasun@yzu.edu.cn}
\author{Hui-Xiang Chen}
\address{School of Mathematical Science, Yangzhou University,
	Yangzhou 225002, China}
\email{hxchen@yzu.edu.cn}
\begin{abstract}
In this article, we investigate the representations of the Drinfeld doubles $D(R_{mn}(q))$ of the Radford Hopf algebras $R_{mn}(q)$ over an algebraically closed field $\Bbbk$, where $m>1$ and $n>1$ are integers
and $q\in\Bbbk$ is a root of unity of order $n$. Under the assumption ${\rm char}(\Bbbk)\nmid mn$, all the finite dimensional indecomposable modules over $D(R_{mn}(q))$ are displayed and classified up to isomorphism. The Auslander-Reiten sequences in the category of finite dimensional $D(R_{mn}(q))$-modules are also all displayed. It is shown that $D(R_{mn}(q))$ is of tame representation type.
\end{abstract}
\maketitle

\section{Introduction}\selabel{1}

Let $m, n\>2$ be integers and $q$ a primitive $n$-th root of unity in a field $\Bbbk$.
Radford \cite{Rad1975} constructed an $mn^2$-dimensional Hopf algebra such that its Jacobson radical is not a Hopf ideal. The Hopf algebra is denoted by $R_{mn}(q)$ and called a Radford Hopf algebra here. Note that the coradical of the dual Hopf algebra $R_{mn}(q)^*$ is not a Hopf subalgebra. When $m=1$, $R_{mn}(q)$ is still well-defined, which is exactly the Taft Hopf algebra $A_n(q)$ of dimension $n^2$ given in \cite{Ta}. Krop and Radford \cite{KropRad} defined the rank as a measure of complexity for Hopf algebras and classified all finite dimensional pointed Hopf algebras of rank one over an algebraically closed field of characteristic zero by using group datums. Scherotzke \cite{Sche} classified such Hopf algebras for the case of char$(\Bbbk)=p>0$. It was proven in \cite{KropRad, Sche} that a finite dimensional pointed Hopf algebra of rank one over an algebraically closed field is isomorphic to a quotient of a Hopf-Ore extension of a group algebra (its coradical). For the Hopf-Ore extension, one can see \cite{Panov}. The Taft algebra $A_n(q)$ and the Radford algebra $R_{mn}(q)$ are both pointed Hopf algebras of rank one.
The finite dimensional pointed rank one Hopf algebras are clarified into two types: nilpotent type and non-nilpotent type \cite{KropRad}. The Taft algebra $A_n(q)$ is of nilpotent type, but the Radford algebra $R_{mn}(q)$ is of non-nilpotent type.

The representations of finite dimensional pointed Hopf algebras of rank one was well studied. For instance, Cibils \cite{Cib} studied the representation theory of (generalized) Taft algebras and described the decomposition rules of tensor product modules over (generalized) Taft algebras. In \cite{ChVOZh} and \cite{LiZhang}, the authors studied respectively the Green rings of the Taft algebras and the generalized Taft algebras. Wang, Li and Zhang studied the representations and Green rings of finite dimensional pointed rank one Hopf algebras of nilpotent type and non-nilpotent type over an algebraically closed field of characteristic zero in \cite{WangLiZhang2014, WangLiZhang2016}, respectively. The representations of Drinfeld doubles of finite dimensional pointed Hopf algebras of rank one also got many achievements. The second author of this paper described the structures of Drinfeld doubles of Taft algebras and classified the finite dimensional indecomposable modules over such Drinfeld doubles in \cite{Ch1, Ch2, Ch3, Ch4}. He with his colleagues studied the Grothendieck rings and Green rings of Drinfeld doubles of Taft algebras, see \cite{Ch5, ChHasSun, ChHasLinSun, SunHasLinCh, ZWLC}.
When $n=2$, the Taft algebra $A_2(-1)$ is the same as $H_4$, the Sweedler's 4-dimensional Hopf algebra (see \cite{Sw, Ta}). The Green ring of Drinfeld double of $H_4$ was also studied by Li and Hu in \cite{LiHu}. Krop and Radford studied the representations of Drinfeld double $D(H)$ of any finite dimensional pointed rank one Hopf algebra $H$ over an algebraically closed field of characteristic zero, and classified the simple modules and projective indecomposable modules over $D(H)$ in \cite{KropRad}, when the group $G(H)$ of group-like elements in $H$ is abelian. However, the other finite dimensional indecomposable modules over $D(H)$ are not classified.

In this article, we study the representations of the Drinfeld doubles $D(R_{mn}(q))$
of Radford Hopf algebras $R_{mn}(q)$ over an algebraically closed field $\Bbbk$ with ${\rm char}(\Bbbk)\nmid mn$. The article is organized as follows. In Section \ref{s2},
we recall the definition of a group datum, the structures of Radford Hopf algebras $R_{mn}(q)$ and their Drinfeld doubles $D(R_{mn}(q))$, denoted by $H_{mn}(\xi)$, where $\xi\in\Bbbk$ is a primitive $mn$-th root of unity with $\xi^m=q$.
In Section \ref{s3}, we recall the simple modules and projective indecomposable modules over $H_{mn}(\xi)$, which shows that the Loewy length of $H_{mn}(\xi)$ is 3. In particular, we reconstruct the non-simple projective indecomposable $H_{mn}(\xi)$-modules and give a standard $\Bbbk$-basis for each such module. In Section \ref{s4}, we study the finite dimensional indecomposable $H_{mn}(\xi)$-modules of Loewy length 2. Using Auslander-Reiten theory, we display all the finite dimensional indecomposable $H_{mn}(\xi)$-modules of Loewy length 2, and classify them up to isomorphism. All the Auslander-Reiten sequences in the category of finite dimensional $H_{mn}(\xi)$-modules are also displayed. It is shown that $H_{mn}(\xi)$ is of tame representation type.

Throughout, let $\Bbbk$ be an algebraically closed field and $\Bbbk^{\times}=\Bbbk\backslash\{0\}$. Unless
otherwise stated, all algebras and Hopf algebras are
defined over $\Bbbk$; all modules are finite dimensional and left modules;
dim and $\otimes$ denote ${\rm dim}_{\Bbbk}$ and $\otimes_{\Bbbk}$,
respectively. Let $\mathbb Z$ denote the set of all integers, ${\mathbb Z}_n:={\mathbb Z}/n{\mathbb Z}$
for an integer $n$, and let $\mathbb{N}$ denote all non-negative integers.
The references \cite{Ka, Mon, Sw} are basic references for the theory of Hopf algebras and quantum groups. The readers can refer \cite{ARS} for the representation theory of algebras.

\section{Drinfeld doubles of Radford Hopf algebras}\label{s2}

In this section, we recall the Radford Hopf algebras and  their Drinfeld doubles.

A quadruple $\mathcal{D}=(G, \chi, g, \mu)$ is called a {\it group datum} if $G$ is a finite group, $\chi$ is a $\Bbbk$-valued character of $G$, $g$ is a central element of $G$ and $\mu\in\Bbbk$ subject to $\chi^n=1$ or $\mu(g^n-1)=0$, where $n$ is the order of $\chi(g)$. The group datum $\mathcal{D}$ is said to be of {\it nilpotent type} if $\mu(g^n-1)=0$, and it is of {\it non-nilpotent type} if $\mu(g^n-1)\neq 0$ and $\chi^n=1$. For any group datum $\mathcal{D}=(G, \chi, g, \mu)$, Krop and Radford constructed an associated finite dimensional pointed rank one Hopf algebra $H_{\mathcal{D}}$ and classified such Hopf algebras. They also described the Drinfeld doubles $D(H_{\mathcal D})$ of $H_{\mathcal D}$, see \cite{KropRad}.

Let $n>1$ and $m>1$ be integers and let $q\in \Bbbk$ be a primitive $n$-th root of unity. Then the Radford Hopf algebra $R_{mn}(q)$ can be described by a group datum as follows.  Let $G$ be a cyclic group of order $mn$ with generator $g$ and let $\chi$ be the $\Bbbk$-valued character of $G$ defined by $\chi(g)=q$. Then $\mathcal{D}=(G, \chi, g, 1)$ is a group datum of non-nilpotent type, and the associated Hopf algebra $H_{\mathcal D}$ is exactly isomorphic to $R_{mn}(q)$. Regarding $R_{mn}(q)=H_{\mathcal D}$ from now on. Then $R_{mn}(q)$ is generated as an algebra by $g$ and $x$ subject to the relations:
$$g^{mn}=1,\quad x^n=g^n-1,\quad xg=qgx.$$
The comultiplication $\t$, counit $\e$ and  antipode $S$ are given by
$$\begin{array}{lll}
\t(x)=x\otimes g+1\otimes x, & \e(x)=0, & S(x)=-xg^{-1},\\
 \t(g)=g\otimes g, & \e(g)=1, &  S(g)=g^{-1}=g^{mn-1}.
\end{array}$$
$R_{mn}(q)$ has a $\Bbbk$-basis $\{g^ix^j|i\in\mathbb{Z}_{mn}, 0\<j\<n-1\}$. For the details, one can refer to \cite{KropRad, Rad1975}.

Let $\xi\in\Bbbk$ be a primitive $mn$-th root of unity with $q=\xi^m$. Then by \cite[Subsection 2.1]{KropRad}, the Drinfeld double $D(R_{mn}(q)))$ can be described as follows. $D(R_{mn}(q)))$ is generated as an algebra by $a, b, c$ and $d$ subject to the following relations:
$$\begin{array}{c}
a^n=b^n-1,\ b^{mn}=1, \ c^{mn}=1,\ d^n=0,\\
ab=qba,\ dc=\xi cd, \ bc=cb,\ bd=qdb, \ ad-da=b-c^m,\\
ac=\xi^{-1}ca+\frac{\xi^{-1}-\xi^{n-1}}{(n-1)!_q}(c^{m+1}-qbc)d^{n-1}.\\
\end{array}$$
The comultiplication, the counit and the antipode of $D(R_{mn}(q)))$ are given by
$$
\begin{array}{c}
\t(a)=a\otimes b+1\otimes a, \quad \t(b)=b\otimes b, \quad \t(d)=c^m\otimes d+d\otimes 1,\\
\t(c)=c\otimes c+(\xi^{n}-1)\sum_{k=1}^{n-1}\frac{1}{(k)!_q(n-k)!_q}c^{mk+1}d^{n-k}\otimes cd^k,\\
  \e(a)=0, \quad \e(b)=1, \quad \e(c)=1, \quad  \e(d)=0,\\
S(a)=-ab^{-1}, \quad S(b)=b^{-1}, \quad S(c)=c^{-1}, \quad S(d)=-c^{-m}d.\\
\end{array}$$

Throughout the following, assume that $\Bbbk$ contains a primitive $mn$-th root $\xi$ of unity with $q=\xi^m$. Then ${\rm char}(\Bbbk)\nmid mn$. Let $H_{mn}(\xi):=D(R_{mn}(q))$ be the Drinfeld double of $R_{mn}(q)$ as described above. Obviously, $H_{mn}(\xi)$ has a $\Bbbk$-basis $\{a^ib^jc^ld^k|j, l\in\mathbb{Z}_{mn},0\<i, k\<n-1\}$, and ${\rm dim}(H_{mn}(\xi))=m^2n^4$.
For simplicity, we denote the image of an integer $i$ under the canonical epimorphism $\mathbb{Z}\ra\mathbb{Z}_{mn}$ still by $i$.

\section{Simple modules and projective modules}\label{s3}

Throughout this and the next sections,
for any $H_{mn}(\xi)$-module $V$, let $P(V)$ and $I(V)$ denote the projective cover and the injective envelope of $V$, respectively. Let $l(V)$ and ${\rm rl}(V)$ denote the length and the radical length (Loewy length) of $V$, respectively. Moreover, let $sV$ denote the direct sum of $s$ copies of $V$ for any integer $s\>0$, where $sV=0$ when $s=0$. For any $H_{mn}(\xi)$-module $M$ and $x_1, \cdots x_s\in M$, let $\langle x_1, \cdots, x_s\rangle$ denote the submodule of $M$ generated by $\{x_1, \cdots, x_s\}$, and define $M^d:=\{v\in M|dv=0\}$.

\subsection{Simple modules}\label{s3.1}

Krop and Radford classified the simple modules over $D(H_{\mathcal D})$ for any group datum $\mathcal{D}=(G,\chi,g,\mu)$ with $G$ being abelian in \cite[Subsection 2.2]{KropRad}. Therefore, one can get the classification of simple modules over $H_{mn}(\xi)$ from \cite[Subsection 2.2]{KropRad}. In this subsection, we recall the simple $H_{mn}(\xi)$-modules.

For any $1\<l\<n-1$ and $i\in \mathbb{Z}$, let $\a_k(l,i)=(k)_q(q^i-q^{i+l-k})$, $1\<k\<l-1$. It is easy to see that $\a_k(l,i)\neq 0$ for all $1\<k\<l-1$. Let
$$I_0=\{(i,j)|i,j\in\mathbb{Z}_{mn}\text{ with }q^{i+k}\neq \xi^j \text{ for all } 0\<k\<n-2\}.$$
Then $\sharp I_0=(mn)^2-mn(n-1)$, where $\sharp I_0$ denotes the number of elements in $I_0$.

Let $U(H_{mn}(\xi))$ be the multiplicative group of the invertible elements of $H_{mn}(\xi)$. Then $b, c\in U(H_{mn}(\xi))$ by $b^{mn}=c^{mn}=1$. Let $G$ be the subgroup of $U(H_{mn}(\xi))$ generated by $b$ and $c$. Then $G=\{b^ic^j|0\<i,j\<mn-1\}$ and $G\cong \mathbb{Z}_{mn}\times\mathbb{Z}_{mn}$ as groups. Let $A$ be the subalgebra of $H_{mn}(\xi)$ generated by $b, c$ and $d$. Then the group algebra $\Bbbk G$ is a subalgebra of $A$ and the Jacobson radical $J(A)$ of $A$ is equal to the principal ideal $(d)$ generated by $d$. Moreover, $A=\Bbbk G\oplus(d)$ as vector spaces. Hence the quotient algebra $A/J(A)$ is exactly isomorphism to the group algebra $\Bbbk G$. Thus, an $A$-module $M$ is semisimple if and only if $dM=0$, and there is a 1-1 correspondence between the simple $A$-modules and the simple $\Bbbk G$-modules. This implies that the simple $A$-modules are in 1-1 correspondence with the $\Bbbk$-valued characters of $G$ since $G$ is abelian. For each character $\l$, let $\Bbbk_{\l}$ denote the corresponding one dimensional $A$-module, and let $Z(\l):=H_{mn}(\xi)\otimes_A\Bbbk_{\l}$ denote the induced $H_{mn}(\xi)$-module. Since $H_{mn}(\xi)$ is a free right $A$-module with an $A$-basis $\{a^i|0\<i\<n-1\}$, $Z(\l)$ is an $n$-dimensional $H_{mn}(\xi)$-module
with a $\Bbbk$-basis $\{a^i\ot_A1_{\l}|0\<i\<n-1\}$, where $1_{\l}$ is a nonzero element of $\Bbbk_{\l}$.

For any $i,j\in \mathbb{Z}$, there is a character $\l_{ij}$ of $G$ given by $\l_{ij}(b)=\xi^j$ and $\l_{ij}(c)=\xi^i$. Clearly, any character of $G$ has this form.
Moreover, for any $i,i',j,j'\in\mathbb{Z}$,
$$\l_{ij}=\l_{i'j'} \Leftrightarrow
i\equiv i'\ (\text{mod }mn) \text{ and } j\equiv j'\ (\text{mod }mn).$$

Let $i,j\in \mathbb{Z}$. Then by \cite[Proposition 6(a)]{KropRad}, $Z(\l_{ij})$ is simple if and only if $\l_{ij}(bc^{-m})=\xi^{j-mi}\notin \{1, q, q^2, \cdots, q^{n-2}\}$ if and only if $(i,j)\in I_0$. In this case, $\{a^k\ot_A1_{\l_{ij}}|0\<k\<n-1\}$ is a $\Bbbk$-basis of $Z(\l_{ij})$. Moreover, a straightforward computation shows that
$Z(\l_{ij})^d=\Bbbk(1\ot_A1_{\l_{ij}})$. Now suppose $(i,j)\notin I_0$. Then there is a unique integer $l$ with $1\<l\<n-1$ such that $\xi^{j-mi}=q^{l-1}$. In this case, $Z(\l_{ij})$ contains a unique proper submodule ${\rm rad}Z(\l_{ij})={\rm span}\{a^k\otimes_A1_{\l_{ij}}|l\<k\<n-1\}$ by \cite[Proposition 6(b)]{KropRad}. Denote the quotient module $Z(\l_{ij})/{\rm rad} Z(\l_{ij})$ by $M(l,i)$. Then $M(l,i)$ is simple. Define $m_k\in M(l,i)$ for $1\<k\<l$ by $m_1=1\ot_A1_{\l_{ij}}+{\rm rad}Z(\l_{ij})$ if $k=1$, and
$m_k=\frac{1}{\a_1\cdots\a_{k-1}}a^{k-1}\ot_A1_{\l_{ij}}+{\rm rad}Z(\l_{ij})$ if $2\<k\<l$, where $\a_k=\a_k(l,i)$ for all $1\<k\<l-1$. Then $\{m_1, m_2, \cdots, m_l\}$ is a $\Bbbk$-basis of $M(l,i)$, and hence ${\rm dim}M(l,i)=l$. A tedious but standard verification shows that the $H_{mn}(\xi)$-module action on $M(l,i)$ is determined by
\begin{equation*}
\begin{array}{ll}
\vspace{0.1cm}
am_k=\left\{
\begin{array}{ll}
\a_k(l,i)m_{k+1},& 1\<k<l-1,\\
0,& k=l,\\
\end{array}\right.&
dm_k=\left\{
\begin{array}{ll}
0, & k=1,\\
m_{k-1},& 1<k\<l,\\
\end{array}\right.\\
bm_k=q^{i+l-k}m_k,\ 1\<k\<l, &cm_k=\xi^{i+k-1}m_k, \ 1\<k\<l.\\
\end{array}
\end{equation*}
Such a basis is called a standard basis of $M(l,i)$. Clearly, $M(l,i)^d=\Bbbk m_1$.

By \cite[Subsection 2.2]{KropRad} and the discussion above, one gets the following lemmas and proposition.

\begin{lemma}\label{3.1}
Let $(i,j), (i',j')\in I_0$.
\begin{enumerate}
\item[(1)]  $Z(\l_{ij})$ is a simple module.
\item[(2)]  $Z(\l_{ij})\cong Z(\l_{i'j'})$ if and only if $(i,j)=(i',j')$.
\end{enumerate}
\end{lemma}

\begin{lemma}\label{3.2}
Let $1\<l, l'\<n-1$ and  $i, i'\in \mathbb{Z}$.
\begin{enumerate}
\item[(1)] $M(l,i)$ is a simple module.
\item[(2)] $M(l,i)\cong M(l',i')$ if and only if $l=l'$ and $i\equiv i'$ ${\rm(mod}\ mn)$.
\item[(3)] $M(l,i+kn)\ncong M(n-l, i+l+k'n)$ for any $k,k'\in\mathbb Z$.
\end{enumerate}
\end{lemma}

\begin{proposition}\label{3.3} The following set
$$\{M(l,i), Z(\l_{i'j'})|1\<l\<n-1,i\in\mathbb{Z}_{mn},(i',j')\in I_0\}$$ is a representative set of isomorphism classes of simple $H_{mn}(\xi)$-modules.
\end{proposition}

By Proposition \ref{3.3} and the structures of the simple $H_{mn}(\xi)$-modules $M(l,i)$ and $Z(\l_{ij})$, one gets the following corollary.

\begin{corollary}\label{3.4}
Let $M$ be a finite dimensional semisimple $H_{mn}(\xi)$-module. Then $l(M)={\rm dim}(M^d)$.
\end{corollary}

\subsection{Projective modules}\label{s3.2}

Krop and Radford described all projective indecomposable $D(H_{\mathcal D})$-modules for any group datum $\mathcal{D}=(G,\chi,g,\mu)$ with $G$ being abelian in \cite[Subsection 2.3]{KropRad}. They described the radical series and simple factors for such modules. Hence form \cite[Subsection 2.3]{KropRad}, one can get the classification of projective indecomposable modules over $H_{mn}(\xi)$.
However, in order to classify all finite dimensional indecomposable $H_{mn}(\xi)$-modules, it is necessary to know the more details of the non-simple projective indecomposable modules over $H_{mn}(\xi)$. For this goal, we investigate non-simple projective indecomposable $H_{mn}(\xi)$-modules in this subsection.

Since $H_{mn}(\xi)$ is a symmetric algebra, $P(V)\cong I(V)$ and $P(V)/{\rm rad}(P(V))\cong{\rm soc}(P(V))\cong V$ for any simple $H_{mn}(\xi)$-module $V$.
Let $J$ denote the Jacobson radical of $H_{mn}(\xi)$. Then
by \cite[Subsection 2.3]{KropRad}, we have the following lemma.

\begin{lemma}\label{3.5}
Let $(i,j)\in I_0$ and $k,l\in\mathbb{Z}$ with $1\<l\<n-1$.
\begin{enumerate}
\item[(1)] $Z(\l_{ij})$ is a projective module.
\item[(2)] ${\rm rl}(P(M(l,k)))=3$ and $l(P(M(l,k)))=4$.
\item[(3)] ${\rm rl}(H_{mn}(\xi))=3$ and hence $J^3=0$.
\end{enumerate}
\end{lemma}

For $1\<l\<n-1$ and $i\in \mathbb{Z}$, define $\g_{k}(l,i), \a_{k}(l,i)\in\Bbbk$, $1\<k\<n$, by
$$\begin{array}{c}
	\vspace{0.1cm}
\g_{k}(l,i)=(k)_q(q^{i+l}-q^{i-k}),\ \a_{k}(l,i)=(k)_q(q^{i}-q^{i+l-k}),\ 1\<k\<n-1,\\
\g_{n}(l,i)=\frac{1}{(n-1)!_q}(q^i-q^{i+l}), \ \a_{n}(l,i)=\frac{1}{(n-1)!_q}(q^{i+l}-q^{i}).\\
\end{array}$$

\begin{lemma}\label{3.6}
Let $1\<l\<n-1$, $i\in\mathbb Z$ and $1\<k\<n$.
\begin{enumerate}
\item[(1)] $\g_{k}(l,i)=0$ if and only if $k=n-l$.
\item[(2)] $\a_{k}(l,i)=0$ if and only if $k=l$.
\item[(3)] $\g_{k}(l,i)=\a_{n-k}(l,i)$ for all $1\<k\<n-1$ and $\g_{n}(l,i)=-\a_{n}(l,i)$.
\item[(4)] $\g_k(l,i)=\a_{l+k}(l,i)$ for all $1\<k\<n-l-1$.
\item[(5)] $\g_k(l,i)=\a_{k-(n-l)}(l,i)$ for all $n-l+1\<k\<n-1$.
\item[(6)] $\g_{k}(l,i)=\a_k(n-l,i+l)$ for all $1\<k\<n$.
\item[(7)] $\g_k(l,i)=\g_k(l,i+n)$ and $\a_k(l,i)=\a_k(l,i+n)$ for all $1\<k\<n$.
\end{enumerate}
\end{lemma}

\begin{proof}
It follows from a straightforward verification.
\end{proof}

For any positive integer $s$, let $M_s(\Bbbk)$ denote the algebra of $s\times s$ matrices over $\Bbbk$. Let $I_s\in M_s(\Bbbk)$ denote the identity matrix, and let $D_s\in M_s(\Bbbk)$ be defined by
$$D_s=\left(
\begin{array}{ccccc}
0 & 1 & 0 & \cdots & 0 \\
0 & 0 & 1 & \cdots & 0 \\
0  & 0 & 0 & \ddots & \vdots \\
\vdots & \vdots & \vdots & \ddots & 1 \\
0 & 0 & 0 & \cdots & 0 \\
\end{array}
\right).$$

For $1\<l\<n-1$ and $i\in \mathbb{Z}$, let $X_{l,i}, Z_{l,i}\in M_{n}(\Bbbk)$ be given by
$$X_{l,i}=\left(
      \begin{array}{ccccc}
        0 & 0 & 0 & \cdots & \g_{n}(l,i) \\
        \g_{1}(l,i) & 0 & 0 & \cdots & 0 \\
        0 & \g_{2}(l,i) & 0 & \cdots & 0 \\
        \vdots & \vdots & \ddots & \ddots & \vdots \\
        0 & 0 & \cdots & \g_{n-1}(l,i) & 0 \\
      \end{array}
    \right),$$
    $$
 \ Z_{l,i}=\left(
      \begin{array}{ccccc}
        0 & 0 & 0 & \cdots & \a_{n}(l,i) \\
        \a_{1}(l,i) & 0 & 0 & \cdots & 0 \\
        0 & \a_{2}(l,i) & 0 & \cdots & 0 \\
        \vdots & \vdots & \ddots & \ddots & \vdots \\
        0 & 0 & \cdots & \a_{n-1}(l,i) & 0 \\
      \end{array}
    \right),
$$

and define $A(l,i)$, $B(l,i)$, $C(l,i)$ and $D$ in $M_{2n}(\Bbbk)$ by
$$\begin{array}{c}
\vspace{0.1cm}
A(l,i)=\left(
    \begin{array}{cc}
      X_{l,i} & 0 \\
      D_n^{n-l-1} & Z_{l,i} \\
    \end{array}
  \right),\
D=\left(
    \begin{array}{cc}
      D_n & 0 \\
      0 & D_n \\
    \end{array}
\right),\\
\vspace{0.1cm}
B(l,i)={\rm diag}\{q^{i-1},q^{i-2},\cdots,q^{i-n},q^{i+l-1},q^{i+l-2},\cdots,q^{i+l-n}\},\\
C(l,i)={\rm diag}\{\xi^{i+l-n},\xi^{i+l-n+1}, \cdots, \xi^{i+l-1}, \xi^i,\xi^{i+1},\cdots, \xi^{i+n-1}\}.\\
\end{array}$$

\begin{lemma}\label{3.7}
Let $1\<l\<n-1$ and $i\in \mathbb{Z}$. Then there is a unique algebra map $\phi_{l,i}:H_{mn}(\xi)\rightarrow M_{2n}(\Bbbk)$ such that
$$\phi_{l,i}(a)=A(l,i),\ \phi_{l,i}(b)=B(l,i),\ \phi_{l,i}(c)=C(l,i),\ \phi_{l,i}(d)=D.$$
Let $P(l,i)$ denote the corresponding left $H_{mn}(\xi)$-module.
\end{lemma}
\begin{proof}
For simplicity, let $\g_j:=\g_j(l,i)$ and $\a_j:=\a_j(l,i)$ for all $1\<j\<n$.
Let
$$A_1(l,i):=\left(
\begin{array}{cc}
\tilde{X}_{l,i} & 0 \\
D_n^{n-l-1} & \tilde{Z}_{l,i} \\
\end{array}
\right)
\text{ and }
A_2(l,i):=\left(
\begin{array}{cc}
\tilde{X}_{l,i} & 0 \\
T & \tilde{Z}_{l,i} \\
\end{array}
\right),$$
where $\tilde{X}_{l,i}=X_{l,i}-\g_nD_n^{n-1}$, $\tilde{Z}_{l,i}=Z_{l,i}-\a_nD_n^{n-1}$ and $T$ is an $n\times n$ matrix with the $(1, n-l)$-entry and $(l+1,n)$-entry to be 1 and the other entries to be zero.
For any $1\<j\neq k\<2n$, let $P_{j,k}\in M_{2n}(\Bbbk)$ be the elementary matrix obtained from $I_{2n}$ by interchanging the $j^{th}$ row and the $k^{th}$ row, and let $P_{j,k}(\g)$ be the elementary matrix obtained from $I_{2n}$ by adding $\g$ times the $k^{th}$ row to the $j^{th}$ row, where $\g\in\Bbbk$.
Define a matrix $P_1\in M_{2n}(\Bbbk)$ by
$$\begin{array}{c}
P_1:=P_{1,n+l+1}(\g_n)P_{2,n+l+2}(\frac{\g_1\g_n}{\a_{l+1}})
P_{3,n+l+3}(\frac{\g_1\g_2\g_n}{\a_{l+1}\a_{l+2}})\cdots
P_{n-l,2n}(\frac{\g_1\g_2\cdots\g_{n-l-1}\g_n}{\a_{l+1}\a_{l+2}\cdots\a_{n-1}}).\\
\end{array}$$
Then $P_1$ is an invertible matrix. One can check that
$P_1^{-1}A(l,i)P_1=A_1(l,i)$ since $\a_n+\frac{\g_1\g_2\cdots\g_{n-l-1}\g_n}{\a_{l+1}\a_{l+2}\cdots\a_{n-1}}=0$
by Lemma \ref{3.6}(3).
Clearly, if $l=1$ then $A_2(l,i)=A_1(l,i)$. For $l>1$, let $\theta_1, \theta_2, \cdots, \theta_{l-1}\in\Bbbk$ be defined by
$$\begin{array}{c}
\theta_1=\frac{1}{\a_{l-1}}, \theta_2=\frac{1+\theta_1\g_{n-2}}{\a_{l-2}},
\theta_3=\frac{1+\theta_2\g_{n-3}}{\a_{l-3}}, \cdots, \theta_{l-1}=\frac{1+\theta_{l-2}\g_{n-l+1}}{\a_1},
\end{array}$$
and define a matrix $P_2\in M_{2n}(\Bbbk)$ by
$$P_2:=P_{n+l-1,n-1}(-\theta_1)P_{n+l-2,n-2}(-\theta_2)\cdots P_{n+1, n-l+1}(-\theta_{l-1}).$$
Then $P_2$ is an invertible matrix and $P_2^{-1}A_1(l,i)P_2=A_2(l,i)$.
Thus, in any case, $A(l,i)$ is similar to $A_2(l,i)$.
Now let $P_3:=P_{n-l+1,n+1}P_{n-l+2,n+2}\cdots P_{n, n+l}\in M_{2n}(\Bbbk)$. Then
$P_3$ is an invertible matrix and
$P_3^{-1}A_2(l,i)P_3=\left(\begin{matrix}
T_1&0\\
0&T_2\\
\end{matrix}\right)$, where $T_1$ and $T_2$ are strictly lower triangular matrices in $M_n(\Bbbk)$. It follows that $A(l,i)^n=0$.

Clearly, $B(l,i)^n=I_{2n}$, and hence $A(l,i)^n=B(l,i)^n-I_{2n}$. It is easy to check that
$B(l,i)^{mn}=C(l,i)^{mn}=I_{2n}$, $D^n=0$, $A(l,i)B(l,i)=qB(l,i)A(l,i)$, $DC(l,i)=\xi C(l,i)D$, $B(l,i)C(l,i)=C(l,i)B(l,i)$ and $B(l,i)D=qDB(l,i)$. Moreover, one can check that $A(l,i)D-DA(l,i)=B(l,i)-C(l,i)^m$ and
$$\begin{array}{rl}
&A(l,i)C(l,i)\\
 =&\xi^{-1} C(l,i)A(l,i)+\frac{1}{(n-1)!_q}(\xi^{-1}-\xi^{n-1})(C(l,i)^{m+1}-qB(l,i)C(l,i))D^{n-1}.
 \end{array}
 $$
 Thus, the proposition follows.
 \end{proof}

Let $1\<l\<n-1$ and $i\in \mathbb{Z}$. Then there is a $\Bbbk$-basis $\{v_1,v_2,\cdots,v_n,u_1,u_2, \cdots,u_n\}$ in $P(l,i)$ such that the $H_{mn}(\xi)$-action is given by
\begin{equation*}
\begin{array}{ll}
\vspace{0.2 cm}
av_k=\left\{
\begin{array}{ll}
\g_{k}(l,i)v_{k+1},& 1\<k\<n-l-1,\\
\g_{k}(l,i)v_{k+1}+u_{k-n+l+1},& n-l\<k\<n-1,\\
\g_{n}(l,i)v_1+u_{l+1},& k=n,\\
\end{array}\right.&\\
\vspace{0.2 cm}
au_k=\left\{
\begin{array}{ll}
\a_{k}(l,i)u_{k+1}, & 1\<k\<n-1,\\
\a_{n}(l,i)u_1,& k=n,\\
\end{array}\right.\\
\end{array}
\end{equation*}
\begin{equation*}
\begin{array}{ll}
\vspace{0.2 cm}
bv_k=q^{i+n-k}v_k,\ 1\<k\<n, &cv_k={\xi}^{i-(n-l)+k-1}v_k,\ 1\<k\<n\\
\vspace{0.2 cm}
bu_k=q^{i+l-k}u_k, \ \ 1\<k\<n, &cu_k={\xi}^{i+k-1}u_k, \ 1\<k\<n,\\
\vspace{0.2 cm}
dv_k=\left\{
\begin{array}{ll}
0, & k=1,\\
v_{k-1},& 2\<k\<n,\\
\end{array}\right.&
du_k=\left\{
\begin{array}{ll}
0, & k=1,\\
u_{k-1},& 2\<k\<n.\\
\end{array}\right.\\
\end{array}
\end{equation*}
Such a basis is called a standard basis of $P(l,i)$. Clearly, $P(l,i)^d=\Bbbk v_1+\Bbbk u_1$.

\begin{proposition}\label{3.8}
Let $1\<l\<n-1$ and $i\in \mathbb{Z}$. Then $P(l,i)\cong P(M(l,i))$.
\end{proposition}
\begin{proof}
Using the standard basis of $P(l,i)$ given above, and putting
$$M:=\langle u_1\rangle={\rm span}\{u_1,u_2,\cdots,u_{l}\}.$$
Then it is easy to see that $M$ is a submodule of $P(l,i)$ and $M\cong M(l,i)$.
Let $N$ be a simple submodule of $P(l,i)$. Then ${\rm dim}(N^d)=1$ by Corollary \ref{3.4}. Let $0\neq z\in N^d$. Then $z=\b_1v_1+\b_2u_1$ by $N^d\subseteq P(l,i)^d$, where $\b_1, \b_2\in\Bbbk$.  If $\b_1\neq 0$, then $bz-q^{i+l-1}z=(q^{i-1}-q^{i+l-1})\b_1v_1$. Hence $v_1\in N$ and so  $a^{n-l}v_1=\g_1(l,i)\g_2(l,i)\cdots\g_{n-l-1}(l,i)u_1\in N$. Thus, $v_1, u_1\in N$ and ${\rm dim}N^d=2$, a contradiction. Hence $\b_1=0$, $\b_2\neq 0$ and $u_1=\b_2^{-1}z\in N$. This implies $M=\langle u_1\rangle\subseteq N$, and so $N=M$
since $M$ and $N$ are both simple. Thus, ${\rm soc}P(l,i)=M\cong M(l,i)$. Hence $P(l,i)$ is isomorphic to a submodule of $P(M(l,i))$ by $P(M(l,i))\cong I(M(l,i))$.
Hence ${\rm dim}((P(l,i)/{\rm soc}P(l,i))^d)=2$, and so
$l({\rm soc}(P(l,i)/{\rm soc}P(l,i)))\<2$ by Corollary \ref{3.4}.
Let $\ol{v}$ denote the image of $v\in P(l,i)$ under the canonical epimorphism $P(l,i)\ra P(l,i)/{\rm soc}P(l,i)$. Then one can see that $M_1:={\rm span}\{\ol{v_1},\ol{v_2},\cdots,\ol{v_{n-l}}\}$ and $M_2:={\rm span}\{\ol{u_{l+1}},\ol{u_{l+2}},\cdots,\ol{u_n}\}$ are submodules of $P(l,i)/{\rm soc}P(l,i)$.
By Lemma \ref{3.6}(4,6,7), one can check that $M_1\cong M(n-l,i+l-n)$ and $M_2\cong M(n-l,i+l)$.
It follows that  ${\rm soc}(P(l,i)/{\rm soc}P(l,i))=M_1\oplus M_2$, and consequently,
${\rm soc}^2P(l,i)={\rm span}\{v_1,v_2, \cdots,v_{n-l}, u_1,u_2,\cdots,u_n\}$.
Then by Lemma \ref{3.6}(5), one gets that $P(l,i)/{\rm soc}^2P(l,i)\cong M(l,i)$.
Thus, ${\rm rl}(P(l,i))=3$ and $l(P(l,i))=4$. Then by Lemma \ref{3.5}(2), $P(l,i)\cong P(M(l,i))$.
\end{proof}

\begin{corollary}\label{3.9}
A representative set of isomorphism classes of indecomposable projective $H_{mn}(\xi)$-modules is given by
 $$\{P(l,i),Z(\l_{i'j'})|1\<l\<n-1,i\in\mathbb{Z}_{mn},(i',j')\in I_0\}.$$
\end{corollary}
\begin{proof}
It follows from Proposition \ref{3.3}, Lemma \ref{3.5}(1) and Proposition \ref{3.8}.
\end{proof}

\begin{corollary}\label{3.10}
Let $1\<l\<n-1$ and $i\in\mathbb{Z}$. Then
$P(l,i)$ is both a projective cover of $M(l,i)$ and an injective envelope of $M(l,i)$.
\end{corollary}
\begin{proof} It follows from the proof of Proposition \ref{3.8}.
\end{proof}

By Proposition \ref{3.8} and its proof, we have the following corollary.

\begin{corollary}\label{3.11}
If $P$ is a non-simple indecomposable projective $H_{mn}(\xi)$-module, then ${\rm rad}P={\rm soc}^2P$ and ${\rm rad}^2 P={\rm soc}P$.
\end{corollary}

Let $\{v_1,\cdots,v_{n}, u_1,u_2,\cdots,u_n\}$ be the standard basis of $P(l,i)$. By the proof of Proposition \ref{3.8}, $\{v_1,\cdots,v_{n-l}, u_1,u_2,\cdots,u_n\}$ is a basis of ${\rm rad}P(l,i)={\rm soc}^2P(l,i)$. Such a basis is called a standard basis of ${\rm rad}P(l,i)$ for later use. By Proposition \ref{3.3}, Lemma \ref{3.5}, Proposition \ref{3.8}, and \cite[Lemma 3.5]{Ch3}, one gets the following lemma.

\begin{lemma}\label{3.12}
Let $M$ be an indecomposable $H_{mn}(\xi)$-module. If ${\rm rl}(M)$=1, then either $M\cong M(l,i)$ for some $1\<l<n$ and $i\in\mathbb{Z}_{mn}$, or $M\cong Z(\l_{ij})$ for some $(i,j)\in I_0$. If ${\rm rl}(M)=3$, then $M\cong P(l,i)$ for some  $1\<l<n$ and $i\in\mathbb{Z}_{mn}$.
\end{lemma}

By Proposition \ref{3.9}, $\mathcal{P}=\{P(l,i),Z(\l_{i'j'})|1\<l\<n-1,i\in\mathbb{Z}_{mn},(i',j')\in I_0\}$ is a complete set of non-isomorphic indecomposable projective $H_{mn}(\xi)$-modules.
Let $\mathcal{P}_{l,i}=\{P(l,i+kn), P(n-l,i+l+kn)|0\<k\<m-1\}$ and $Q_{i',j'}=\{Z(\l_{i'j'})\}$, where $1\<l\<n-1$, $i\in\mathbb{Z}$ and $(i',j')\in I_0$.
Then $\mathcal{P}_{l,i}=\mathcal{P}_{l,i+n}$ and $\mathcal{P}_{l,i}=\mathcal{P}_{n-l,i+l}$ for all $1\<l\<n-1$ and $i\in\mathbb{Z}$. Thus, by Proposition \ref{3.3}, Corollary \ref{3.9} and the proof of Proposition \ref{3.8}, one gets the block partition of indecomposable projective $H_{mn}(\xi)$-modules below.

\begin{corollary}\label{3.13}
\begin{enumerate}
\item[(1)] If $n$ is odd then $\mathcal{P}=(\cup_{1\<l\<\frac{n-1}{2},0\<i\<n-1}\mathcal{P}_{l,i})
    \cup(\cup_{(i,j)\in I_0}Q_{i,j})$ is the block partition of indecomposable projective $H_{mn}(\xi)$-modules.
\item[(2)] If $n$ is even then
 $\mathcal{P}=(\cup_{1\<l\<\frac{n-2}{2},0\<i\<n-1}\mathcal{P}_{l,i})
\cup(\cup_{0\<i\<\frac{n-2}{2}}\mathcal{P}_{\frac{n}{2},i})\cup(\cup_{(i,j)\in I_0}Q_{i,j})$ is the block partition of indecomposable projective $H_{mn}(\xi)$-modules.

\end{enumerate}
\end{corollary}

\section{\bf Indecomposable $H_{mn}(\xi)$-modules with Loewy length 2}\label{s4}

In this section, we investigate the non-simple non-projective indecomposable $H_{mn}(\xi)$-modules. Such indecomposable modules have Loewy length 2.

\begin{proposition}\label{4.1}
Let $M$ be an indecomposable $H_{mn}(\xi)$-module with ${\rm rl}(M)=2$. Then there are  integers $1\<l\<n-1$ and $0\<i\<n-1$ such that ${\rm soc}M\cong\oplus_{k=0}^{m-1}s_kM(l,i+kn)$ and $I(M)\cong\oplus_{k=0}^{m-1}s_kP(l,i+kn)$ for some $s_k\in \mathbb{N}$, $0\<k\<m-1$.
\end{proposition}
\begin{proof}
Since $M$ is an indecomposable $H_{mn}(\xi)$-module with ${\rm rl}(M)=2$, it follows from Corollary \ref{3.13} that there are integers $1\<l\<n-1$ and $0\<i\<n-1$ such that $M$ is an indecomposable module over the block $\mathcal{P}_{l,i}$ of $H_{mn}(\xi)$. Since $H_{mn}(\xi)$ is symmetric, $I(M)$ is projective. Hence $I(M)=P_1\oplus P_2$, where
$P_1\cong\oplus_{k=0}^{m-1}s_kP(l,i+kn)$ and $P_2\cong\oplus_{k=0}^{m-1}t_kP(n-l,i+l+kn)$
for some $s_k, t_k\in\mathbb N$, $0\<k\<m-1$. Regarding $M\subseteq I(M)$,
and putting $M_1=M\cap P_1$ and $M_2=M\cap P_2$.
Since $I(M)$ is an injective envelope of $M$, ${\rm soc}M={\rm soc}I(M)={\rm soc}P_1\oplus{\rm soc}P_2$. Hence ${\rm soc}M\subset M_1\oplus M_2$. Since rl$(M)=2$, $M/{\rm soc}M$ is a semisimple submodule of $I(M)/{\rm soc}I(M)$, and so $M/{\rm soc}M\subseteq{\rm soc}(I(M)/{\rm soc}I(M))={\rm soc}^2I(M)/{\rm soc}I(M)$. By Corollary \ref{3.11}, ${\rm soc}I(M)={\rm rad}^2I(M)={\rm rad}^2P_1\oplus{\rm rad}^2P_2$ and ${\rm soc}^2I(M)={\rm rad}I(M)={\rm rad}P_1\oplus{\rm rad}P_2$. Hence $M\subseteq{\rm rad}I(M)$.
Let $\pi:I(M)\rightarrow I(M)/{\rm soc} I(M)$ be the canonical epimorphism. Then $I(M)/{\rm soc}I(M)=I(M)/{\rm rad}^2I(M)=\pi(P_1)\oplus\pi(P_2)$ and $M/{\rm soc}M\subseteq{\rm rad}I(M)/{\rm rad}^2I(M)=\pi({\rm rad}P_1)\oplus\pi({\rm rad}P_2)$. By the proof of Proposition \ref{3.8}, one gets
$$\begin{array}{rcl}
\pi({\rm rad}P_1)&=&({\rm rad}P_1+{\rm rad}^2I(M))/{\rm rad}^2(I(M))\\
&\cong&{\rm rad}P_1/{\rm rad}^2P_1\cong\oplus_{k=0}^{m-1}(s_k+s_{k+1})M(n-l,i+l+kn),\\
\pi({\rm rad}P_2)&=&({\rm rad}P_2+{\rm rad}^2I(M))/{\rm rad}^2(I(M))\\
&\cong&{\rm rad}P_2/{\rm rad}^2P_2\cong\oplus_{k=0}^{m-1}(t_{k-1}+t_k)M(l,i+kn),\\
\end{array}$$
where $s_m=s_0$ and $t_{-1}=t_{m-1}$. By Proposition \ref{3.2}(3), for any simple submodule $X$ of $M/{\rm soc}M$, either $X\subseteq\pi({\rm rad}P_1)$ or $X\subseteq\pi({\rm rad}P_2)$. If $X\subseteq\pi({\rm rad}P_1)$ then $\pi^{-1}(X)\subseteq M\cap({\rm rad}P_1+{\rm rad}^2I(M))=M\cap({\rm rad}P_1\oplus{\rm rad}^2P_2)=(M\cap{\rm rad}P_1)\oplus{\rm rad}^2P_2\subseteq M_1\oplus M_2$. Similarly, if
$X\subseteq\pi({\rm rad}P_2)$ then $\pi^{-1}(X)\subseteq M_1\oplus M_2$.
It follows that $M\subseteq M_1\oplus M_2$, and so $M=M_1\oplus M_2$. However, $M$ is indecomposable, we have $M_1=0$ or $M_2=0$, and consequently $P_1=0$ or $P_2=0$.
Thus either $I(M)\cong\oplus_{k=0}^{m-1}s_kP(l,i+kn)$ and ${\rm soc}M\cong\oplus_{k=0}^{m-1}s_kM(l,i+kn)$, or $I(M)\cong\oplus_{k=0}^{m-1}t_kP(n-l,i+l+kn)$
and ${\rm soc}M\cong\oplus_{k=0}^{m-1}t_kM(n-l,i+l+kn)$. This completes the proof.
\end{proof}

Recall the syzygy functor  $\Omega$ and the cosyzygy functor $\Omega^{-1}$. Let $M$ be an $H_{mn}(\xi)$-module. Choose a fixed projective cover $f: P(M)\ra M$ and define the syzygy $\Omega M$ to be ${\rm Ker}f$, called the syzygy of $M$. Choose an injective envelope $f: M\ra I(M)$ and define the cosyzygy $\Omega^{-1}M$ to be ${\rm Coker}f$.
If $M$ has no nonzero projective (injective) direct summands, then neither do $\Omega M$ and $\Omega^{-1}M$, and $\Omega\Omega^{-1}M\cong M\cong\Omega^{-1}\Omega M$. Moreover, $M$ is indecomposable if and only if $\Omega M$ is indecomposable. If $N$ is also an $H_{mn}(\xi)$-module without nonzero projective (injective) direct summands, then $M\cong N$ if and only if $\Omega M\cong\Omega N$, cf. \cite[p.126]{ARS}.

Let $M$ be an indecomposable $H_{mn}(\xi)$-module with ${\rm rl}(M)=2$. By \cite[Lemma 3.7]{Ch3},
${\rm soc}M={\rm rad}M$. If $l({\rm soc}M)=t$ and $l(M/\text {soc}M)=s$, then we say that $M$ is of $(s,t)$-type (cf. \cite{Ch3}).

\begin{lemma}\label{4.2}
Let $M$ be an indecomposable $H_{mn}(\xi)$-module with ${\rm rl}(M)=2$.
\begin{enumerate}
\item[(1)] $M$ is of $(1,2)$-type $\Leftrightarrow$ $M\cong \Omega^{-1} M(l,i)$ for some $1\<l\<n-1$ and $i\in \mathbb{Z}_{mn}$.
\item[(2)] $M$ is of $(2,1)$-type $\Leftrightarrow$ $M\cong \Omega M(l,i)$ for some $1\<l\<n-1$ and $i\in \mathbb{Z}_{mn}$.
\end{enumerate}
\end{lemma}
\begin{proof}
(1) By Corollary \ref{3.10} and the proof of Proposition \ref{3.8}, if $M\cong\Omega^{-1}M(l,i)$ for some $1\<l\<n-1$ and $i\in \mathbb{Z}_{mn}$,  then $M$ is of $(1,2)$-type. Conversely, let $M$ is of $(1,2)$-type. Then
$M/{\rm rad}M=M/{\rm soc}M\cong M(l,i)$ for some $1\<l\<n-1$ and $i\in \mathbb{Z}_{mn}$.
It follows from Corollary \ref{3.10} that there is an $H_{mn}(\xi)$-module epimorphism $f: P(l,i)\ra M$. Since $l(M)=3$ and $l(P(l,i))=4$, $l({\rm Ker}f)=1$. Hence ${\rm Ker}f={\rm soc}P(l,i)\cong M(l,i)$, and so $M\cong P(l,i)/{\rm Ker}f=P(l,i)/{\rm soc}P(l,i)\cong\Omega^{-1}M(l,i)$ by Corollary \ref{3.10}.

(2) It is similar to (1) or dual to (1).
\end{proof}

\begin{lemma}\label{4.3}
Let $M$ be of $(s,t)$-type.
\begin{enumerate}
\item[(1)] $s\<2t$ and $t\<2s$.
\item[(2)] If $s\neq 1$, then $t<2s$ and $\Omega M$ is of $(2s-t,s)$-type.
\item[(3)] If $t\neq 1$, then $s<2t$ and $\Omega^{-1} M$ is of $(t,2t-s)$-type.
\item[(4)] If $s=t$, then both $\Omega M$ and $\Omega^{-1} M$ are of $(t,t)$-type.
\end{enumerate}
\end{lemma}
\begin{proof}
By Proposition \ref{4.1}, there are integers $1\<l\<n-1$ and $0\<i\<n-1$ such that ${\rm soc}M\cong\oplus_{k=0}^{m-1}t_kM(l,i+kn)$ and $I(M)\cong\oplus_{k=0}^{m-1}t_kP(l,i+kn)$ for some $t_k\in \mathbb{N}$, $0\<k\<m-1$.
Moreover, $\sum_{k=0}^{m-1}t_k=t$.

(1) By the proof of Proposition \ref{4.1}, $M/{\rm soc}M\subseteq{\rm rad}I(M)/{\rm rad}^2I(M)$ and
$${\rm rad}I(M)/{\rm rad}^2I(M)\cong\oplus_{k=0}^{m-1}(t_k+t_{k+1})M(n-l,i+l+kn),$$
where $t_m=t_0$. Hence $s=l(M/{\rm soc}M)\<l({\rm rad}I(M)/{\rm rad}^2I(M))=2t$.
Furthermore, we have $M/{\rm rad}M=M/{\rm soc}M\cong\oplus_{k=0}^{m-1}s_kM(n-l,i+l+kn)$ for some $s_k\in \mathbb{N}$, $0\<k\<m-1$, with $\sum_{k=0}^{m-1}s_k=s$.
Hence $P=\oplus_{k=0}^{m-1}s_kP(n-l,i+l+kn)$ is a projective cover of $M$. Thus, there is an epimorphism $f: {\rm rad}P\rightarrow{\rm rad}M$. Since ${\rm rad}M={\rm soc}M$ is semisimple, ${\rm rad}^2P\subseteq{\rm Ker}f$, and hence $l({\rm soc}M)\<l({\rm rad}P/{\rm rad}^2P)=2s$, and so $t\<2s$.

(2) Assume $s\neq1$. By the proof of $(1)$, $f: P\rightarrow M$ is a projective cover with ${\rm soc}P={\rm rad}^2P\subseteq{\rm Ker}f$. Hence $f$ induces an epimorphism
$\ol{f}: P/{\rm rad}^2P\rightarrow M$. If $t=2s$, then $l(M)=3s=l(P/\text{soc}P)$, and hence $\ol{f}$ is an isomorphism. However, $P/\text{rad}^2P$ is not indecomposable by $s>1$, a contradiction. Therefore, $t<2s$. Meanwhile, $\Omega M\cong {\rm Ker}f\subseteq{\rm rad}P$, ${\rm soc}P={\rm rad}^2P\subset{\rm Ker}f$ and ${\rm rad}^2P\neq{\rm Ker}f$. It follows that ${\rm rl}(\Omega M)=2$ and $l(\text{soc}(\Omega M))=l(\text{soc}P)=s$. But $l(\Omega M)=l(P)-l(M)=3s-t$ and so $l(\Omega M/\text{soc}(\Omega M))=2s-t$. That is, $\Omega M$ is of $(2s-t,s)$-type.

(3) It is dual to (2).

(4) It is clear.
\end{proof}

\begin{corollary}\label{4.4}
Let $M$ be of $(s,t)$-type with ${\rm soc}M\cong\oplus_{k=0}^{m-1}t_kM(l,i+kn)$ for some integers $1\<l\<n-1$, $0\<i\<n-1$ and $t_k\>0$.
\begin{enumerate}
\item[(1)] $M/{\rm soc}M\cong\oplus_{k=0}^{m-1}s_kM(n-l,i+l+kn)$ for some integers $s_k\>0$.
\item[(2)] ${\rm Hom}_{H_{mn}(\xi)}(M, {\rm soc}M)={\rm Hom}_{H_{mn}(\xi)}(M/{\rm soc}M, {\rm soc}M)=0$.
\end{enumerate}
\end{corollary}
\begin{proof}
(1) follows from the proof of Proposition \ref{4.1}. Since ${\rm rad}M={\rm soc}M$ and ${\rm soc}M$ is semisimple, it follows from (1) and Lemma \ref{3.2}(3) that
$$\begin{array}{rl}
&{\rm Hom}_{H_{mn}(\xi)}(M, {\rm soc}M)={\rm Hom}_{H_{mn}(\xi)}(M/{\rm soc}M, {\rm soc}M)\\
\cong&{\rm Hom}_{H_{mn}(\xi)}(\oplus_{k=0}^{m-1}s_kM(n-l,i+l+kn), \oplus_{k=0}^{m-1}t_kM(l,i+kn))=0.\\
\end{array}$$
This shows (2).
\end{proof}

\begin{proposition}\label{4.5}
Let $M$ be an indecomposable $H_{mn}(\xi)$-module of $(s,t)$-type.
\begin{enumerate}
\item[(1)] If $s<t$, then $t=s+1$ and $M\cong\Omega^{-s}M(l,i)$ for some $1\<l\<n-1$ and $i\in \mathbb{Z}_{mn}$.
\item[(2)] If $s>t$, then $s=t+1$ and $M\cong\Omega^tM(l,i)$ for some $1\<l\<n-1$ and $i\in \mathbb{Z}_{mn}$.
\end{enumerate}
\end{proposition}
\begin{proof}

$(1)$ If $1=s<t$, then $t=2$ by Lemma \ref{4.3}(1) and $M\cong \Omega^{-1}M(l,i)$ for some $1\<l\<n-1$ and $i\in \mathbb{Z}_{mn}$ by Lemma \ref{4.2}. Now assume $1<s<t$. Then $t<2s$ by Lemma \ref{4.3}. Set $r=t-s$, then $1\<r<s$. We know that there is a unique positive integer $l'$ such that $l'r<s\<(l'+1)r$. We claim that $\Omega^iM$ is of $(s-ir,s-(i-1)r)$-type for all $1\<i\<l'$. In fact, $\Omega M$ is of $(s-r,s)$-type by Lemma \ref{4.3}(2). Now let $1\<i<l'$ and assume $\Omega^iM$ is of $(s-ir,s-(i-1)r)$-type. Then $s-ir>s-l'r\>1$ and $\Omega^{i+1}M$ is of $(s-(i+1)r , s-ir)$-type again by Lemma \ref{4.3}(2). Thus we have proved the claim. In particular, $\Omega^{l'}M$ is of $(s-l'r,s-(l'-1)r)$-type. If $s-l'r>1$, then $s-(l'-1)r<2(s-l'r)$ by Lemma \ref{4.3}(2), which forces $(l'+1)r<s$, a contradiction. Hence $s-l'r=1$ and so $r=1$. It follows that $t=s+1$, $s=l'+1$, and $\Omega^{l'}M$ is of $(1,2)$-type. Thus, by Lemma \ref{4.2}, $\Omega^{s-1}M\cong \Omega^{-1}M(l,i)$ for some $1\<l\<n-1$ and $i\in \mathbb{Z}_{mn}$, and so $M\cong \Omega^{-s}M(l,i)$.

$(2)$ It is similar to (1) or dual to (1).
\end{proof}

\begin{corollary}\label{4.6}
Let $M$ be the indecomposable $H_{mn}(\xi)$-module with rl$(M)=2$. Then either $M$ is of $(s,s)$-type or $M$ is isomorphic to $\Omega^{\pm s}M(l,i)$ for some integers $s\>1$ and  $1\<l\<n-1$, and $i\in \mathbb{Z}_{mn}$.
\end{corollary}
\begin{proof}
It follows from Proposition \ref{4.5}.
\end{proof}

\begin{lemma}\label{4.7}
For any $1\<l<n$, $i\in \mathbb{Z}_{mn}$, we have following Auslander-Reiten sequences in ${\rm mod}H_{mn}(\xi)$:
\begin{enumerate}
\item[(1)] $0\rightarrow \Omega M(l,i)\rightarrow M(n-l,i+l)\oplus M(n-l,i+l-n)\oplus P(l,i)\rightarrow \Omega^{-1}M(l,i)\rightarrow 0$;
\item[(2)] $0\rightarrow \Omega^{t+2} M(l,i)\rightarrow \Omega^{t+1}M(n-l,i+l)\oplus \Omega^{t+1}M(n-l,i+l-n)\rightarrow \Omega^tM(l,i)\rightarrow 0$;
\item[(3)] $0\rightarrow \Omega^{-t} M(l,i)\rightarrow \Omega^{-(t+1)}M(n-l,i+l)\oplus \Omega^{-(t+1)}M(n-l,i+l-n)\rightarrow \Omega^{-(t+2)}M(l,i)\rightarrow 0$,
    \end{enumerate}
where $t\>0$ and $\Omega^0M(l,i)=M(l,i)$.
\end{lemma}
\begin{proof}
It is similar to \cite[Theorem 3.17]{Ch3}.
\end{proof}

In what follows, we investigate the indecomposable $H_{mn}(\xi)$-modules of $(t,t)$-type.

\begin{lemma}\label{4.8}
Let $M$ be an indecomposable modules of $(t,t)$-type with $t\>2$. Then $M$ contains no submodules of $(s+1,s)$-type. Consequently, $M$ contains no submodules $N$ with $l(N/{\rm soc}N)>l({\rm soc}N)$.
\end{lemma}
\begin{proof}

We first show the result for $s=1$. By Proposition \ref{4.1}, we may assume ${\rm soc}M\cong\oplus_{k=0}^{m-1}t_kM(l,i+kn)$ for some $1\<l\<n-1$, $0\<i\<n-1$ and $t_k\in\mathbb{N}$ with $\sum_{k=0}^{m-1}t_k=t$. If $M$ contains a submodule $N$ of $(2,1)$-type, then $N\cong \Omega M(l,i+kn)\cong{\rm rad}P(l,i+kn)$ for some fixed $k$ with $t_k>0$ by Lemma \ref{4.2} and ${\rm soc}N\subset{\rm soc}M$. Hence, there is a short exact sequence
$$0\rightarrow {\rm rad}P(l,i+kn)\rightarrow M\rightarrow K\rightarrow 0$$ with $K\neq 0$ since $l({\rm rad}P(l,i+kn))=3<l(M)$. By Corollary \ref{4.4}(2), we have
${\rm Hom}_{H_{mn}(\xi)}(K, M(l,i+kn))=0$. On the other hand, there is anther short exact sequence
$$0\rightarrow {\rm rad}P(l,i+kn)\rightarrow P(l,i+kn)\rightarrow M(l,i+kn)\rightarrow 0,$$ form which one obtains a long exact sequence
$$\begin{array}{rl}
0&\rightarrow {\rm Hom}_{H_{mn}(\xi)}(K,{\rm rad}P(l,i+kn))\rightarrow{\rm Hom}_{H_{mn}(\xi)}(K,P(l,i+kn))\\
 &\rightarrow{\rm Hom}_{H_{mn}(\xi)}(K,M(l,i+kn))\rightarrow {\rm Ext}^1_{H_{mn}(\xi)}(K,{\rm rad}P(l,i+kn))\rightarrow 0.
 \end{array}$$
Hence ${\rm Ext}^1_{H_{mn}(\xi)}(K,{\rm rad}P(l,i+kn))=0$, which implies $M\cong {\rm rad}P(l,i+kn)\oplus K$, a contradiction. This shows that $M$ contains no submodules of $(2,1)$-type. Now suppose $1<s<t$ and any indecomposable module of $(t,t)$-type contains no submodules of $(j+1,j)$-type for all $1\<j<s$. If $M$ contains a submodule $N$ of $(s+1,s)$-type, then ${\rm soc}N\cong\oplus_{k=0}^{m-1}s_kM(l,i+kn)$ for some $s_k\in\mathbb N$ with $s_k\<t_k$ and $\sum_{k=0}^{m-1}s_k=s$. By Corollary \ref{4.4}(1), $N/{\rm soc}N\cong \oplus_{k=0}^{m-1}s'_kM(n-l,i+l+kn)$ for some $s'_k\in\mathbb{N}$ with $\sum_{k=0}^{m-1}s'_k=s+1$. Moreover, there is a short exact sequence. $$0\rightarrow N\rightarrow M\xrightarrow{\pi}M/N\rightarrow 0.$$
Let $V$ be a simple submodule of $M/N$. Then $N\subset\pi^{-1}(V)\subseteq M$ and $\pi^{-1}(V)/N\cong V$. If $V$ is not isomorphic to any submodule of ${\rm soc}M$,
then $V$ is not isomorphic to any submodule of ${\rm soc}(\pi^{-1}(V))$, and hence ${\rm soc}(\pi^{-1}(V))={\rm soc}N$. In this case, $l({\rm soc}(\pi^{-1}(V)))=s$ and
$l(\pi^{-1}(V)/{\rm soc}(\pi^{-1}(V)))=s+2$. Thus, $\pi^{-1}(V)$ is decomposable by Lemma \ref{3.12} and Corollary \ref{4.6}, and $\pi^{-1}(V)$ has at least one direct summand of $(j+1,j)$-type with $j<s$, which contradicts the induction hypothesis. This shows that $V$ is isomorphic to a submodule of ${\rm soc}M$. Then by Corollary \ref{4.4}, one can check that $V\subseteq\pi({\rm soc}M)$, ${\rm soc}(M/N)=\pi({\rm soc}M)\cong\oplus_{k=0}^{m-1}r_kM(l,i+kn)$ for some $r_k\in \mathbb{N}$ with $\sum_{k=0}^{m-1}r_k=t-s$ and ${\rm soc}M\cong{\rm soc}N\oplus{\rm soc}(M/N)$. Hence $I(M)\cong I(N)\oplus I(M/N)$. Thus, any monomophism $M\rightarrow I(N)\oplus I(M/N)$ is an envelope of $M$. Now by applying $\Omega^{-1}$ to the above exact sequence, one gets an exact sequence $$0\rightarrow \Omega^{-1}(N)\rightarrow \Omega^{-1}(M)\rightarrow \Omega^{-1}(M/N)\rightarrow 0.$$ This contradicts the induction hypothesis since $\Omega^{-1}(M)$ is of $(t,t)$-type and $\Omega^{-1}(N)$ is of $(s,s-1)$-type by Lemma \ref{4.3}(3). Therefore, $M$ contains no submodule of $(s+1,s)$-type. This completes the proof.
\end{proof}

\begin{lemma}\label{4.9}
Assume that $M$ be an indecomposable modules of $(t,t)$-type. Let $L$ and $N$ be submodules of $M$.
If $l(L/{\rm soc}L)=l({\rm soc}L)$ and $l(N/{\rm soc}N)=l({\rm soc}N)$, then $l((L\cap N)/{\rm soc}(L\cap N))=l({\rm soc}(L\cap N))$.
\end{lemma}

\begin{proof}
Suppose $l(L/{\rm soc}L)=l({\rm soc}L)$ and $l(N/{\rm soc}N)=l({\rm soc}N)$. For any submodule $X$ of $M$, we have
$[X]=[X/{\rm soc}X]+[{\rm soc}X]$ in $G_0(H_{mn}(\xi))$. Moreover, $X/{\rm soc}X\cong(X+{\rm soc}M)/{\rm soc}M\subseteq M/{\rm soc}M$ and ${\rm soc}X\subseteq{\rm soc}M$. Since $[L+N]=[L]+[N]-[L\cap N]$ in $G_0(H_{mn}(\xi)$, it follows from Corollary \ref{4.4}(2) that
\begin{eqnarray*}
	&[(L+N)/{\rm soc}(L+N)]=[L/{\rm soc}L]+[N/{\rm soc}N]-[(L\cap N)/{\rm soc}(L\cap N)],\\
	&[{\rm soc}(L+N)]=[{\rm soc}L]+[{\rm soc}N]-[{\rm soc}(L\cap N)]
\end{eqnarray*}
in $G_0(H_{mn}(\xi))$. Hence we have
\begin{eqnarray*}
&l((L+N)/{\rm soc}(L+N))=l(L/{\rm soc}L)+l(N/{\rm soc}N)-l((L\cap N)/{\rm soc}(L\cap N)),\\
&l({\rm soc}(L+N))=l({\rm soc}L)+l({\rm soc}N)-l({\rm soc}(L\cap N)).
\end{eqnarray*}
By Lemma  \ref{4.8}, $l((L+N)/{\rm soc}(L+N))\<l({\rm soc}(L+N))$ and $l((L\cap N)/{\rm soc}(L\cap N))\<
l({\rm soc}(L\cap N))$. This forces $l((L\cap N)/{\rm soc}(L\cap N))=l({\rm soc}(L\cap N))$.
\end{proof}

For any positive integer $i, j$. Let $T_{i,j}\in M_{i\times j}(\Bbbk)$ be given by
$$\left(
    \begin{array}{ccccc}
      0 & 0 & \cdots & 0 & 1 \\
      0 & 0 & \cdots & 0 & 0 \\
       & \cdots & \cdots & \cdots &  \\
      0 & 0 & \cdots & 0 & 0 \\
    \end{array}
  \right).$$
Clearly, $T_{s,s}=D_s^{s-1}$ for any integer  $s\>1$. In particular, $T_{n,n}=D_n^{n-1}$.
For any $1\<l\<n-1$ and $i\in \mathbb{Z}$, define $A_l(i)$ in $M_{l}(\Bbbk)$ by
$$\begin{array}{c}
\vspace{0.1cm}
A_l(i)=\left(
      \begin{array}{ccccc}
        0 & 0 & 0 & \cdots & 0 \\
        \a_1(l,i) & 0 & 0 & \cdots & 0 \\
        0 & \a_2(l,i) & 0 & \cdots & 0 \\
        \vdots & \vdots & \ddots & \ddots & \vdots \\
        0 & 0 & \cdots & \a_{l-1}(l,i) & 0 \\
      \end{array}
    \right),\\
\end{array}$$
where $\a_k(l,i)$ is defined in the last section. By Lemma \ref{3.6}(2), $\a_k(l,i)\neq 0$ for all $1\<k\<l-1$.
\begin{lemma}\label{4.10}
$(1)$ For any $1\<l\<n-1$ and $i\in \mathbb{Z}$, there is a unique algebra morphism $\phi: H_{mn}(\xi)\rightarrow M_n(\Bbbk)$ such that
$$\begin{array}{c}
\phi(a)=Z_{l,i},\ \phi(b)={\rm diag}\{q^{i+l-1},q^{i+l-2},\cdots,q^{i+l-n}\},\\
\phi(c)={\rm diag}\{\xi^i,\xi^{i+1},\cdots, \xi^{i+n-1}\},\ \phi(d)=D_n,\\
\end{array}$$
where $Z_{l,i}$ and $D_n$ are given in Section 3. Denote by $T_1(l,i)$ the corresponding left $H_{mn}(\xi)$-module. Then $T_1(l,i)$ is of $(1,1)$-type, ${\rm soc}T_1(l,i)\cong M(l,i)$ and $T_1(l,i)/{\rm soc}T_1(l,i)\cong M(n-l,i+l)$.

$(2)$ For any $1\<l\<n-1$ and $i\in \mathbb{Z}$, there is a unique algebra morphism $\phi: H_{mn}(\xi)\rightarrow M_n(\Bbbk)$ such that
$$\begin{array}{c}
\phi(a)=\left(
           \begin{array}{cc}
             A_{n-l}(i+l-n) & 0 \\
             T_{l,n-l} & A_l(i) \\
           \end{array}
         \right),
\ \phi(d)=\left(
             \begin{array}{cc}
               D_{n-l} & 0 \\
               0 & D_l \\
             \end{array}
           \right),\\
\phi(b)={\rm diag}\{q^{i-1},q^{i-2},\cdots, q^{i-n}\},\
\phi(c)={\rm diag}\{\xi^{i+l-n},\xi^{i+l-n+1},\cdots, \xi^{i+l-n+n-1}\}.\\
\end{array}$$
Denote by $\ol{T}_1(l,i)$ the corresponding left $H_{mn}(\xi)$-module. Then ${\ol T}_1(l,i)$ is of $(1,1)$-type, and ${\rm soc}{\ol T}_1(l,i)\cong M(l,i)$ and $\ol{T}_1(l,i)/{\rm soc}\ol{T}_1(l,i)\cong M(n-l,i+l-n)$.
\end{lemma}
\begin{proof}
It follows from a straightforward computation.
\end{proof}

Obviously, we have the following lemma.

\begin{lemma}\label{4.11}
Let $1\<l, l'\<n-1$ and $i,i'\in\mathbb{Z}$.
\begin{enumerate}
\item[(1)] $T_1(l,i)\ncong\ol{T}_1(l',i')$.
\item[(2)] $T_1(l,i)\cong T_1(l',i')$ if and only if $l=l'$ and $i\equiv i'$ $({\rm mod}\ mn)$.
\item[(3)] $\ol{T}_1(l,i)\cong\ol{T}_1(l',i')$ if and only if $l=l'$ and $i\equiv i'$ $({\rm mod}\ mn)$.
\end{enumerate}
\end{lemma}

\begin{proposition}\label{4.12}
Let $M$ be of $(1,1)$-type. Then there are $1\<l\<n-1$ and $i\in \mathbb{Z}$ such that $M\cong T_1(l,i)$ or $M\cong\ol{T}_1(l,i)$.
\end{proposition}
\begin{proof}
Since $M$ is of $(1,1)$-type, there are $1\<l\<n-1$ and $i\in \mathbb{Z}$ such that
${\rm soc}M\cong M(l,i)$. Hence $P(l,i)$ is an injective envelope of $M$. We may assume that $M$ is a submodule of $P(l,i)$. Then ${\rm soc}M={\rm soc}P(l,i)$, $M/{\rm soc}M\subset{\rm soc}^2P(l,i)/{\rm soc}P(l,i)$ and $M\subset{\rm soc}^2P(l,i)$.
Let $v_1,v_2,\cdots, v_n, u_1,u_2,\cdots,u_n$ be the standard basis of $P(l,i)$. Then by the proof of Proposition \ref{3.8}, ${\rm soc}P(l,i)={\rm span}\{u_1,u_2,\cdots,u_l\}\cong M(l,i)$ and ${\rm soc}^2P(l,i)/{\rm soc}P(l,i)=M_1\oplus M_2$, where
$$\begin{array}{l}
M_1:={\rm span}\{v_j+{\rm soc}P(l,i)|1\<j\<n-l\}\cong M(n-l,i+l-n),\\
M_2:={\rm span}\{u_j+{\rm soc}P(l,i)|l+1\<j\<n\}\cong M(n-l,i+l).\\
\end{array}$$
By Lemma \ref{3.2}, $M_1$ and $M_2$ are simple submodules of ${\rm soc}^2P(l,i)/{\rm soc}P(l,i)$ and $M_1\ncong M_2$. Since $M/{\rm soc}M$ is a simple submodule of ${\rm soc}^2P(l,i)/{\rm soc}P(l,i)$, $M/{\rm soc}M=M_1$ or $M/{\rm soc}M=M_2$.

If $M/{\rm soc}M=M_1$, then $M={\rm span}\{v_1,v_2,\cdots,v_{n-l},u_1,u_2,\cdots,u_l\}$. In this case, by Lemma \ref{3.6}(6,7), one can check that $M\cong\ol{T}_1(l,i)$.
If $M/{\rm soc}M=M_2$, then $M={\rm span}\{u_1,u_2,\cdots,u_n\}$. In this case, one can check that $M\cong T_1(l,i)$.
\end{proof}

\begin{remark}\label{4.13}
Let $1\<l\<n-1$ and $i\in \mathbb{Z}$. By Proposition \ref{4.12} and its proof, we have
$$\begin{array}{ll}
\Omega^{-1}T_1(l,i)\cong T_1(n-l,i+l-n),& \Omega T_1(l,i)\cong T_1(n-l,i+l),\\
\Omega^{-1}\ol{T}_1(l,i)\cong\ol{T}_1(n-l,i+l),& \Omega\ol{T}_1(l,i)\cong\ol{T}_1(n-l,i+l-n).\\
\end{array}$$
Hence $\Omega^2T_1(l,i)\cong T_1(l,i+n)$ and $\Omega^2\ol{T}_1(l,i)\cong\ol{T}_1(l,i-n)$. Moreover, one can check that ${\rm End}_{H_{mn}(\xi)}(T_1(l,i))\cong\Bbbk$, ${\rm End}_{H_{mn}(\xi)}(\ol{T}_1(l,i))\cong\Bbbk$ and
$${\rm Hom}_{H_{mn}(\xi)}(T_1(l,i),T_1(l,i+kn))=0,\ {\rm Hom}_{H_{mn}(\xi)}(\ol{T}_1(l,i),\ol{T}_1(l,i+kn))=0,$$
for all $k\in\mathbb Z$ with $m\nmid k$.
\end{remark}

\begin{lemma}\label{4.14}
Let $M$ and $N$ be nonzero indecomposable $H_{mn}(\xi)$-modules and assume there is an irreducible morphism $f:M\rightarrow N$. Then $M$ is of $(t,t)$-type for some $t$ if and only if $N$ is of $(s,s)$-type for some $s$.
\end{lemma}
\begin{proof}
It follows from Lemma \ref{3.12}, Corollary \ref{4.6}, Lemma \ref{4.7} and \cite[Theorem V.5.3]{ARS}.
\end{proof}

Let $1\<l\<n-1$ and $i\in \mathbb{Z}$. Since $DTr T_1(l,i-n)\cong \Omega^2 T_1(l,i-n)\cong T_1(l,i)$, one gets an Auslander-Reiten sequence (AR-sequence for short)
\begin{eqnarray*}
0\rightarrow T_1(l,i)\rightarrow M\rightarrow T_1(l,i-n)\rightarrow 0.
\end{eqnarray*}
If $M$ has a nontrivial indecomposable direct summand $N$, then there is an irreducible morphism $f: T_1(l,i)\rightarrow N$. By Lemma \ref{4.14} and $l(M)=4$, $N$ is of $(1,1)$-type. Any irreducible morphism is either a monomorphism or an epimorphism, but it is not an isomorphism. However, $l(T_1(l,i))=l(N)$. It is absurd. Hence $M$ is an indecomposable module of $(2,2)$-type. Denote $M$ by $T_2(l,i)$. Thus, we have an AR-sequence
\begin{eqnarray}\eqlabel{e1}
0\rightarrow T_1(l,i)\xrightarrow{f_{1}} T_2(l,i)\xrightarrow{g_{1}} T_1(l,i-n)\rightarrow 0.
\end{eqnarray}
Clearly, ${\rm soc}T_2(l,i)\cong M(l,i)\oplus M(l,i-n)$ and $T_2(l,i)/{\rm soc}T_2(l,i)\cong M(n-l,i+l)\oplus M(n-l,i+l-n)$.

\begin{lemma}\label{4.15}
${\rm Im}f_{1}$ is the only submodule of $(1,1)$-type contained in $T_2(l,i)$.
\end{lemma}
\begin{proof}
Let $N$ be a submodule of $(1,1)$-type of $T_2(l,i)$. Then by Lemma \ref{4.9},
$$l((N\cap{\rm Im}f_1)/{\rm soc}(N\cap{\rm Im}f_1))=l({\rm soc}(N\cap{\rm Im}f_1)).$$
Hence either $N\cap{\rm Im}f_1=0$ or $N\cap{\rm Im}f_1=N$. If $N\cap{\rm Im}f_1=0$, then $T_2(l,i)=N\oplus{\rm Im}f_1$, a contradiction. Hence $N\cap{\rm Im}f_1=N$, and so $N={\rm Im}f_1$.
\end{proof}

By applying $\Omega^2$ to the sequence \equref{e1}, one gets another AR-sequence
\begin{eqnarray*}
0\rightarrow\Omega^2T_1(l,i)\rightarrow\Omega^2T_2(l,i)\rightarrow\Omega^2 T_1(l,i-n)\rightarrow 0.
\end{eqnarray*}
Hence $\Omega^2T_2(l,i)\cong T_2(l,i+n)$ by $\Omega^2T_1(l,i)\cong T_1(l,i+n)$ and $\Omega^2T_1(l,i-n)\cong T_1(l,i)$.
By \cite[Proposiotns V.1.12 and 5.3]{ARS}, there is a unique module $T_3(l,i)$ up to isomorphism such that there is an AR-sequence
\begin{eqnarray*}
0\rightarrow T_2(l,i)\xrightarrow{\left(
                                     \begin{matrix}
                                       g_1 \\
                                       f_2 \\
                                     \end{matrix}
                                   \right)
} T_1(l,i-n)\oplus T_3(l,i)\xrightarrow{(f'_1, g_2)} T_2(l,i-n)\rightarrow 0.
\end{eqnarray*}

\begin{lemma}\label{4.16}
$T_3(l,i)$ is an indecomposable module of $(3,3)$-type with ${\rm soc}T_3(l,i)\cong\oplus_{k=0}^2M(l,i-kn)$, $T_3(l,i)/{\rm soc}T_3(l,i)\cong\oplus_{k=0}^2M(n-l,i+l-kn)$ and $\Omega^2 T_3(l,i)\cong T_3(l,i+n)$. Moreover, ${\rm Im}(f_2f_1)$ and ${\rm Im}f_2$ are only submodules of $(1,1)$-type and $(2,2)$-type, respectively, and $T_3(l,i)/{\rm Im}(f_2f_1)\cong T_2(l,i-n)$ and $T_3(l,i)/{\rm Im}f_2\cong T_1(l,i-2n)$.
\end{lemma}
\begin{proof}
It is similar to \cite[Lemma 4.10]{Ch3} by using Lemmas \ref{4.9}, \ref{4.14} and \ref{4.15}.
\end{proof}

Now suppose $t\>3$ and we have got indecomposable modules
$$T_1(l,i), T_2(l,i),\cdots,  T_t(l,i)$$
and AR-sequences
\begin{eqnarray*}
&0\rightarrow T_1(l,i)\xrightarrow{f_1}T_2(l,i)\xrightarrow{g_1} T_1(l,i-n)\rightarrow 0,\\
&0\rightarrow T_{s-1}(l,i)\xrightarrow{\left(\begin{matrix}
                                       g_{s-2} \\
                                       f_{s-1} \\
                                     \end{matrix}\right)}
T_{s-2}(l,i-n)\oplus T_s(l,i)\xrightarrow{(f'_{s-2},g_{s-1})}
T_{s-1}(l,i-n)\rightarrow 0
\end{eqnarray*}
for all $3\<s\<t$, where $1\<l\<n-1$ and $i\in\mathbb{Z}$. They satisfy the following:
\begin{enumerate}
\item[(1)] $T_s(l,i)$ is of $(s,s)$-type with ${\rm soc}T_s(l,i)\cong\oplus_{k=0}^{s-1}M(l,i-kn)$ and\\
     $T_s(l,i)/{\rm soc}T_s(l,i)\cong \oplus_{k=0}^{s-1}M(n-l,i+l-kn)$;
\item[(2)] ${\rm Im}(f_{s-1}\cdots f_j$) is the only submodule of $(j,j)$-type of $T_s(l,i)$ and\\
    $T_s(l,i)/{\rm Im}(f_{s-1}\cdots f_j)\cong T_{s-j}(l,i-jn)$ for all $1\<j<s$;
\item[(3)] $\Omega^2T_s(l,i)\cong T_s(l,i+n)$,
\end{enumerate}
for all $1\<s\<t$. Again by \cite[Proposiotns V.1.12 and 5.3]{ARS}, there is a unique module $T_{t+1}(l,i)$ up to isomorphism such that there is an AR-sequence
\begin{eqnarray*}
0\rightarrow T_t(l,i)\xrightarrow{\left(
                                     \begin{matrix}
                                       g_{t-1} \\
                                       f_t \\
                                     \end{matrix}
                                   \right)
} T_{t-1}(l,i-n)\oplus T_{t+1}(l,i)\xrightarrow{(f'_{t-1}, g_t)} T_t(l,i-n)\rightarrow 0.
\end{eqnarray*}
Then similarly to Lemma \ref{4.16}, we have the following lemma.

\begin{lemma}\label{4.17}
$T_{t+1}(l,i)$ is an indecomposable module of $(t+1,t+1)$-type with ${\rm soc}T_{t+1}(l,i)\cong \oplus_{k=0}^tM(l,i-kn)$, $T_{t+1}(l,i)/{\rm soc}T_{t+1}(l,i)\cong \oplus_{k=0}^tM(n-l,i+l-kn)$ and $\Omega^2T_{t+1}(l,i)\cong T_{t+1}(l,i+n)$. Moreover, ${\rm Im}(f_t\cdots f_j)$ is the unique submodule of $(j,j)$-type of $T_{t+1}(l,i)$ and $T_{t+1}(l,i)/{\rm Im}(f_t\cdots f_j)\cong T_{t+1-j}(l,i-jn)$ for all $1\<j\<t$.
\end{lemma}
\begin{proof}
It is similar to \cite[Lemma 4.11]{Ch3} by using Corollary \ref{4.4}, Lemmas \ref{4.8}, \ref{4.9} and \ref{4.14}.
\end{proof}

Summarizing the discussion above, one gets the following theorem.

\begin{theorem}\label{4.18}
For any $1\<l\<n-1$ and $i, t\in \mathbb{Z}$ with $t\>1$, there is an indecomposable $H_{mn}(\xi)$-module $T_t(l,i)$ of $(t,t)$-type. We have following properties:
\begin{enumerate}
\item[(1)] $\Omega^2T_t(l,i)\cong T_t(l,i+n)$,  ${\rm soc}T_t(l,i)\cong\oplus_{k=0}^{t-1}M(l,i-kn)$ and\\
$T_t(l,i)/{\rm soc}T_t(l,i)\cong\oplus_{k=0}^{t-1}M(n-l,i+l-kn)$.
\item[(2)] For any $1\<j<t$, $T_t(l,i)$ contains a unique submodule of $(j,j)$-type, which is isomorphic to $T_j(l,i)$ and the quotient module of $T_t(l,i)$ modulo the submodule of $(j,j)$-type is isomorphic to $T_{t-j}(l,i-jn)$.
\item[(3)] For any $1\<j<t$, the unique submodule of $(j,j)$-type of $T_t(l,i)$ is contained in that of $(j+1,j+1)$-type.
\item[(4)] There are AR-sequences
\begin{eqnarray*}
&0\ra T_1(l,i)\xrightarrow{f_1}
T_2(l,i)\xrightarrow{g_1}
T_1(l,i-n)\ra 0,\\
&0\ra T_t(l,i)\xrightarrow{\left(\begin{matrix}
                                       g_{t-1} \\
                                       f_t \\
                                     \end{matrix}\right)}
T_{t-1}(l,i-n)\oplus T_{t+1}(l,i)\xrightarrow{(f'_{t-1}, g_t)}
T_t(l,i-n)\ra 0\ (t\>2).
\end{eqnarray*}
\end{enumerate}
\end{theorem}

Starting from $\ol{T}_1(l,i)$, a similar discussion as above shows the following theorem.

\begin{theorem}\label{4.19}
For any $1\<l\<n-1$ and $i, t\in \mathbb{Z}$ with $t\>1$, there is an indecomposable $H_{mn}(\xi)$-module $\ol{T}_t(l,i)$ of $(t,t)$-type. We have following properties:
\begin{enumerate}
\item[(1)] $\Omega^2\ol{T}_t(l,i)\cong\ol{T}_t(l,i-n)$, ${\rm soc}\ol{T}_t(l,i)\cong \oplus_{k=0}^{t-1}M(l,i+kn)$ and\\ $\ol{T}_t(l,i)/{\rm soc}\ol{T}_t(l,i)\cong \oplus_{k=0}^{t-1}M(n-l,i+l+(k-1)n)$.
\item[(2)] For any $1\<j<t$, $\ol{T}_t(l,i)$ contains a unique submodule of $(j,j)$-type, which is isomorphic to $\ol{T}_j(l,i)$ and the quotient module of $\ol{T}_t(l,i)$ modulo the submodule of $(j,j)$-type is isomorphic to $\ol{T}_{t-j}(l,i+jn)$.
\item[(3)] For any $1\<j<t$, the unique submodule of $(j,j)$-type of $\ol{T}_t(l,i)$ is contained in that of $(j+1,j+1)$-type.
\item[(4)] There are AR-sequences
\begin{eqnarray*}
&0\ra\ol{T}_1(l,i)\xrightarrow{f_1}\ol{T}_2(l,i)\xrightarrow{g_1}
\ol{T}_1(l,i+n)\ra 0,\\
&0\ra\ol{T}_{t}(l,i)\xrightarrow{\left(\begin{matrix}
                                       g_{t-1} \\
                                       f_t \\
                                     \end{matrix}\right)}
\ol{T}_{t-1}(l,i+n)\oplus\ol{T}_{t+1}(l,i)\xrightarrow{(f'_{t-1}, g_t)}
\ol{T}_{t}(l,i+n)\ra 0\ (t\>2).
\end{eqnarray*}
\end{enumerate}
\end{theorem}

\begin{corollary}\label{4.20}
Let $1\<l,l'\<n-1$ and $i,i',t,t'\in\mathbb{Z}$ with $t\>1$ and $t'\>1$.
\begin{enumerate}
\item[(1)] $T_t(l,i)\ncong\ol{T}_{t'}(l',i')$.
\item[(2)] $T_t(l,i)\cong T_{t'}(l',i')$ if and only if $t=t'$, $l=l'$ and $i\equiv i'$ $({\rm mod}\ mn)$.
\item[(3)] $\ol{T}_t(l,i)\cong\ol{T}_{t'}(l',i')$ if and only if $t=t'$, $l=l'$ and $i\equiv i'$ $({\rm mod}\ mn)$.
\end{enumerate}
\end{corollary}

\begin{proof}
It follows from Lemma \ref{4.11}, Theorems \ref{4.18}(2) and \ref{4.19}(2),   $l(T_t(l,i))=2t$ and $l(\ol{T}_{t'}(l',i'))=2t'$.
\end{proof}

\begin{proposition}\label{4.21}
Let $M$ be an indecomposable $H_{mn}(\xi)$-module of $(t,t)$-type. If $M$ contains a submodule of $(1,1)$-type, then $M\cong T_t(l,i)$ or $M\cong\ol{T}_t(l,i)$ for some $1\<l\<n-1$ and $i\in \mathbb{Z}$.
\end{proposition}

\begin{proof}
If $t=1$, it follows from Proposition \ref{4.12}. Now let $t>1$ and assume that $M$ contains a submodule $N$ of $(1,1)$-type. Then by Proposition \ref{4.12}, $N\cong T_1(l,i)$ or $N\cong\ol{T}_1(l,i)$ for some $1\<l\<n-1$ and $i\in \mathbb{Z}$. If $N\cong T_1(l,i)$, then an argument  similar to the proof of \cite[Theorem 4.16]{Ch3} shows that $M\cong T_t(l,i)$. Similarly, if $N\cong\ol{T}_1(l,i)$ then $M\cong\ol{T}_t(l,i)$.
\end{proof}

For any $H_{mn}(\xi)$-module $M$ and $i\in\mathbb Z$,
let $M_{(i)}:=\{x\in M|cx=\xi^ix\}$.

\begin{lemma}\label{4.22}
Let $M$ be an $H_{mn}(\xi)$-module.
\begin{enumerate}
\item[(1)] For any $i, j\in\mathbb Z$, $M_{(i)}=M_{(j)}\Leftrightarrow i\equiv j$ $({\rm mod}\ mn)$, and hence $M=\oplus_{i\in\mathbb{Z}_{mn}}M_{(i)}$ as vector spaces.
\item[(2)] $dM_{(i)}\subseteq M_{(i-1)}$ for any $i\in\mathbb{Z}$, and hence $M^d=\oplus_{i\in\mathbb{Z}_{mn}}M^d\cap M_{(i)}$.
\item[(3)] If $K$ and $N$ are submodules of $M$ such that $M=K\oplus N$, then
$M_{(i)}=K_{(i)}\oplus N_{(i)}$ for all $i\in\mathbb Z$.
\item[(4)] If $f: M\ra N$ is an $H_{mn}(\xi)$-module map, then $f(M_{(i)})\subseteq N_{(i)}$ for all $i\in\mathbb Z$. Furthermore, if $f$ is surjective, then $f(M_{(i)})=N_{(i)}$ for all $i\in\mathbb Z$.
\end{enumerate}
\end{lemma}

\begin{proof}
It follows from a straightforward verification.
\end{proof}

Let $1\<l\<n-1$, $i\in\mathbb Z$ and $P=\oplus_{k=0}^{m-1}P(l,i+kn)$. For any $0\<k\<m-1$, let
$\{v_1^k, v_2^k, \cdots, v_n^k, u_1^k, u_2^k, \cdots, u_n^k\}$ be a standard basis of $P(l,i+kn)$. Then
$\{v_j^k, u_j^k|1\<j\<n, 0\<k\<m-1\}$ is a basis of $P$. Let $\eta\in{\Bbbk}^{\times}$. For any $1\<j\<n$ and $0\<k\<m-1$, define $x_j^k\in P$
by $x_j^k=u_j^k$ if $1\<j\<l$; $x_j^k=u_j^k+v_{j-l}^{k+1}$ if $l+1\<j\<n$ and $0\<k<m-1$;  $x_j^{m-1}=u_j^{m-1}+\eta v_{j-l}^0$ if $l+1\<j\<n$ and $k=m-1$.
Note that $q^{j+kn}=q^j$ for any $j, k\in\mathbb Z$. By Lemma \ref{3.6}(4), a straightforward verification shows  that
\begin{eqnarray}
\eqlabel{e2} &ax_j^k=\left\{\begin{array}{ll}
\a_j(l,i)x_{j+1}^k,& 1\<j<n, 0\<k\<m-1,\\
\a_n(l,i)x_1^k+x_1^{k+1},& j=n, 0\<k<m-1,\\
\a_n(l,i)x_1^{m-1}+\eta x_1^0,& j=n, k=m-1,\\
\end{array}\right.\\
\eqlabel{e3} &bx_j^k=q^{i+l-j}x_j^k,\ 1\<j\<n, 0\<k\<m-1,\\
\eqlabel{e4} &cx_j^k=\xi^{i+kn+j-1}x_j^k,\  1\<j\<n, 0\<k\<m-1,\\
\eqlabel{e5} &dx_j^k=\left\{\begin{array}{ll}
0,& j=1, 0\<k\<m-1,\\
x_{j-1}^k,& 1<j\<n, 0\<k\<m-1.\\
\end{array}\right.
\end{eqnarray}
This implies that ${\rm span}\{x_j^k|1\<j\<n, 0\<k\<m-1\}$ is a submodule of $P$, denoted by $M_1(l,i,\eta)$.
Clearly, $\{x_j^k|1\<j\<n, 0\<k\<m-1\}$ is a $\Bbbk$-basis of $M_1(l,i,\eta)$. A basis of $M_1(l,i,\eta)$ satisfying Eqs.\equref{e2}-\equref{e5} is called a standard basis.

\begin{lemma}\label{4.23}
Retain the above notations.
\begin{enumerate}
\item[(1)] ${\rm rl}(M_1(l,i,\eta))=2$, ${\rm soc}M_1(l,i,\eta)\cong\oplus_{k=0}^{m-1}M(l,i+kn)$ and\\
  $M_1(l,i,\eta)/{\rm soc}M_1(l,i,\eta)\cong\oplus_{k=0}^{m-1}M(n-l,i+l+kn)$.
\item[(2)] $M_1(l,i,\eta)$ is an indecomposable module of $(m,m)$-type.
\item[(3)] $M_1(l,i,\eta)$ does not contain any submodule of $(1,1)$-type.
\item[(4)] $M_1(l,i,\eta)\ncong T_m(l',i')$ and $M_1(l,i,\eta)\ncong\ol{T}_m(l',i')$ for any $1\<l'\<n-1$ and $i'\in\mathbb Z$
\end{enumerate}
\end{lemma}

\begin{proof}
(1) By the proof of Proposition \ref{3.8}, we have $M_1(l,i,\eta)\subseteq{\rm rad}P$ and
$$\begin{array}{rl}
&{\rm soc}M_1(l,i,\eta)={\rm rad}M_1(l,i,\eta)={\rm soc}P\\
=&{\rm span}\{x_j^k|1\<j\<l, 0\<k\<m-1\}\cong\oplus_{k=0}^{m-1}M(l,i+kn).
\end{array}$$
Hence $M_1(l,i,\eta)/{\rm soc}M_1(l,i,\eta)\subseteq P/{\rm soc}P$. Let $\pi: P\ra P/{\rm soc}P$ be the canonical epimorphiosm.
By Lemma \ref{3.6}(4,6), a straightforward verification shows that ${\rm span}\{\pi(x_{l+1}^k), \cdots, \pi(x_n^k)\}$ is a submodule of $M_1(l,i,\eta)/{\rm soc}M_1(l,i,\eta)$ and isomorphic to $M(n-l,i+l+kn)$ for any $0\<k\<m-1$. Therefore, $$\begin{array}{rl}
M_1(l,i,\eta)/{\rm soc}M_1(l,i,\eta)&={\rm span}\{\pi(x_j^k)|l+1\<j\<n, 0\<k\<m-1\}\\
&\cong\oplus_{k=0}^{m-1}M(n-l,i+l+kn),\\
\end{array}$$
and so ${\rm rl}(M_1(l,i,\eta))=2$.

(2) Suppose $M_1(l,i,\eta)=N\oplus K$ for some submodules $N$ and $K$ of $M_1(l,i,\eta)$. Then by Lemma \ref{4.22}(3), $M_1(l,i,\eta)_{(j)}=N_{(j)}\oplus K_{(j)}$ for any $j\in\mathbb Z$.
It is easy to check that for $1\<j, j'\<n$ and $0\<k,k'\<m-1$,
$$i+kn+j-1\equiv i+k'n+j'-1\ ({\rm mod}\ mn) \Leftrightarrow j=j'\ \text{and } k=k'.$$
Then by Eq.\equref{e4}, $x_j^k\in M_1(l,i,\eta)_{(i+kn+j-1)}$
and ${\rm dim}(M_1(l,i, \eta)_{(i+kn+j-1)})=1$ for all $1\<j\<n$ and $0\<k\<m-1$, and so
$x_j^k\in N$ or $x_j^k\in K$. In particular,
$x_n^0\in N$ or $x_n^0\in K$. Without losing generality, we may assume $x_n^0\in N$. Then
$ax_n^0=\a_n(l,i)x_1^0+x_1^1\in N$, and $cax_n^0=\a_n(l,i)cx_1^0+cx_1^1=\a_n(l,i)\xi^ix_1^0+\xi^{i+n}x_1^1\in N$.
This implies that $x_1^0, x_1^1\in N$. If $x_n^1\in K$, then a similar argument shows that $x_1^1, x_1^2\in K$, and so
$x_1^1\in N\cap K$, a contradiction. Therefore, $x_n^1\in N$. Similarly, one can check that $x_n^2, \cdots, x_n^{m-1}\in N$. However, $M_1(l,i,\eta)=\langle x_n^0, x_n^1, \cdots, x_n^{m-1}\rangle$. Hence $N=M_1(l,i,\eta)$, and so $M_1(l,i,\eta)$ is indecomposable. Then by (1), $M_1(l,i,\eta)$ is of $(m,m)$-type.

(3) Suppose on the contrary that $M_1(l,i,\eta)$ contains a submodule $N$ of $(1,1)$-type. Then by (1), $\pi(N)$ is a simple submodule of $M_1(l,i,\eta)/{\rm soc}M_1(l,i,\eta)$, and hence $\pi(N)\cong  M(n-l,i+l+kn)$ for some $0\<k\<m-1$.
This implies that $\pi(N)=\oplus_{j=l}^{n-1}\pi(N)_{(i+kn+j)}$ and ${\rm dim}(\pi(N)_{(i+kn+j)})=1$ for all $l\<j\<n-1$.
By Lemma \ref{4.22}(4), $\pi(N)_{(i+kn+j)}=\pi(N_{(i+kn+j)})$, and hence $N_{(i+kn+n-1)}\neq 0$. By the proof of (2), $N_{(i+kn+n-1)}=M_1(l,i,\eta)_{(i+kn+n-1)}$ and so $x_n^k\in N$. Then similarly to (2), one gets $x_1^k, x_1^{k+1}\in N$, and hence $\langle x_1^k, x_1^{k+1}\rangle\subseteq N$, where $x_1^m=x_1^0$.
However, $\langle x_1^k, x_1^{k+1}\rangle\cong M(l,i+kn)\oplus M(l,i+(k+1)n)$ and $l({\rm soc}N)=1$, a contradiction. Therefore, $M_1(l,i,\eta)$ does not contain any submodule of $(1,1)$-type.

(4) It follows from (3) and Theorems \ref{4.18}(2) and \ref{4.19}(2).
\end{proof}

\begin{lemma}\label{4.24}
Let $1\<l\<n-1$, $i\in\mathbb Z$ and $\eta\in\Bbbk^{\times}$.
\begin{enumerate}\eqlabel{e6}
\item[(1)]  $\O^{-1}M_1(l,i,\eta)\cong M_1(n-l, i+l-n, (-1)^m\eta)$
\item[(2)] $\O^{-2}M_1(l,i,\eta)\cong M_1(l, i-n, \eta)$ and $\O^2M_1(l,i,\eta)\cong M_1(l, i+n, \eta)$.
\end{enumerate}
\end{lemma}

\begin{proof}
With the notations above, $P=\oplus_{k=0}^{m-1}P(l,i+kn)$ is an injective envelope of $M_1(l,i,\eta)$ and $M_1(l,i,\eta)\subseteq P$. For any $x\in P$, let $\ol{x}$ denote the image of $x$ under the canonical epimorphism $P\ra P/M_1(l,i,\eta)$. Then
\begin{eqnarray}\eqlabel{e6}
\ol{u_j^k}=\left\{\begin{matrix}
\vspace{0.1cm}
-\ol{v_{j-l}^{k+1}}, & l+1\<j\<n, 0\<k<m-1,\\
-\eta\ol{v_{j-l}^0}, & l+1\<j\<n, k=m-1.\\
\end{matrix}\right.
\end{eqnarray}
Let $y_j^k=(-1)^k\ol{v_j^k}$ for all $1\<j\<n$ and $0\<k\<m-1$. Then $\{y_j^k|1\<j\<n, 0\<k\<m-1\}$ is a basis of $P/M_1(l,i,\eta)$. By Lemma \ref{3.6}(6,7) and Eq.\equref{e6},
a straightforward verification shows that
\begin{eqnarray*}
&ay_j^k=\left\{\begin{array}{ll}
\a_j(n-l,i+l-n)y_{j+1}^k,& 1\<j<n, 0\<k\<m-1,\\
\a_n(n-l,i+l-n)y_1^k+y_1^{k+1},& j=n, 0\<k<m-1,\\
\a_n(n-l,i+l-n)y_1^{m-1}+(-1)^m\eta y_1^0,& j=n, k=m-1,\\
\end{array}\right.\\
&by_j^k=q^{i+l-n+(n-l)-j}y_j^k,\ 1\<j\<n, 0\<k\<m-1,\\
&cy_j^k=\xi^{i+l-n+kn+j-1}y_j^k,\  1\<j\<n, 0\<k\<m-1,\\
&dy_j^k=\left\{\begin{array}{ll}
0,& j=1, 0\<k\<m-1,\\
y_{j-1}^k,& 1<j\<n, 0\<k\<m-1.\\
\end{array}\right.
\end{eqnarray*}
It follows that $\O^{-1}M_1(l,i,\eta)\cong P/M_1(l,i,\eta)\cong M_1(n-l,i+l-n,(-1)^m\eta)$.
This shows (1). (2) follows from (1).
\end{proof}

\begin{lemma}\label{4.25}
Let $1\<l\<n-1$, $i\in\mathbb Z$ and $\eta\in\Bbbk^{\times}$. Then
$$M_1(l,i,\eta)\cong M_1(l,i+n,\eta).$$
\end{lemma}

\begin{proof}
Let $\{x_j^k|1\<j\<n, 0\<k\<m-1\}$ and $\{y_j^k|1\<j\<n, 0\<k\<m-1\}$ be the standard bases of $M_1(l,i,\eta)$ and $M_1(l,i+n,\eta)$, respectively. Define a linear map $f: M_1(l,i+n,\eta)\ra M_1(l,i,\eta)$ by
\begin{eqnarray*}
f(y_j^k)=\left\{\begin{array}{ll}
x_j^{k+1}, & 1\<j\<n, 0\<k<m-1,\\
\eta x_j^0, & 1\<j\<n, k=m-1.\\
\end{array}\right.
\end{eqnarray*}
Then $f$ is a linear isomorphism. By a standard computation, one can check that
\begin{eqnarray*}
f(ay_j^k)=af(y_j^k), f(by_j^k)=bf(y_j^k), f(cy_j^k)=cf(y_j^k), f(dy_j^k)=df(y_j^k),
\end{eqnarray*}
where $1\<j\<n$ and $0\<k\<m-1$. Hence $f$ is an $H_{mn}(\xi)$-module isomorphism.
This completes the proof.
\end{proof}

\begin{corollary}\label{4.26}
Let $1\<l\<n-1$, $i\in\mathbb Z$ and $\eta\in\Bbbk^{\times}$. Then
\begin{eqnarray*}
\O^2M_1(l,i,\eta)\cong M_1(l, i, \eta)\cong\O^{-2}M_1(l,i,\eta).
\end{eqnarray*}
\end{corollary}

\begin{proof}
It follows from Lemmas \ref{4.24}(2) and \ref{4.25}.
\end{proof}

\begin{proposition}\label{4.27}
Let $1\<l, l'\<n-1$, $i, i'\in\mathbb Z$ and $\eta, \eta'\in\Bbbk^{\times}$. Then
$$M_1(l,i,\eta)\cong M_1(l', i', \eta') \Leftrightarrow l=l', \eta=\eta' \text{ and }i\equiv i'\ ({\rm mod}\ n).$$
\end{proposition}

\begin{proof}
If $l=l'$, $\eta=\eta'$ and $i\equiv i'\ ({\rm mod}\ n)$, then $M_1(l,i,\eta)\cong M_1(l', i', \eta')$ by Lemma \ref{4.25}.
Conversely, assume $M_1(l,i,\eta)\cong M_1(l', i', \eta')$. Then ${\rm soc}M_1(l,i,\eta)\cong{\rm soc}M_1(l', i', \eta')$.
By Lemma \ref{4.23}(1), $M(l,i)\cong M(l',i'+kn)$ for some $0\<k\<m-1$. It follows from Lemma \ref{3.2}(2) that
$l=l'$ and $mn|i-i'-kn$. This implies $n|i-i'$. Then by Lemma \ref{4.25}, $M_1(l',i',\eta')\cong M_1(l, i, \eta')$, and so
$M_1(l,i,\eta)\cong M_1(l, i, \eta')$. Let $f: M_1(l,i,\eta)\ra M_1(l, i, \eta')$ be a module isomorphism.
Let $\{x_j^k|1\<j\<n, 0\<k\<m-1\}$ and $\{y_j^k|1\<j\<n, 0\<k\<m-1\}$ be the standard bases of $M_1(l,i,\eta)$ and $M_1(l, i, \eta')$, respectively. Then by the proof of Lemma \ref{4.23}(2), $M_1(l,i,\eta)_{(i+kn+j-1)}=\Bbbk x_j^k$ and $M_1(l,i,\eta')_{(i+kn+j-1)}=\Bbbk y_j^k$ for all $1\<j\<n$ and $0\<k\<m-1$. Hence by Lemma \ref{4.22}(4),
there are scalars $\b_0, \b_1, \cdots, \b_{m-1}\in\Bbbk^{\times}$ such that $f(x_n^k)=\b_ky_n^k$ for all $0\<k\<m-1$. For any $1\<j<n$ and $0\<k\<m-1$, $f(d^{n-j}x_n^k)=d^{n-j}f(x_n^k)$, i.e. $f(x_j^k)=\b_ky_j^k$. Then for $0\<k<m-1$, by  $f(ax_n^k)=af(x_n^k)$, one gets that
$\a_n(l,i)\b_ky_1^k+\b_{k+1}y_1^{k+1}=\a_n(l,i)\b_ky_1^k+\b_ky_1^{k+1}$, which implies $\b_k=\b_{k+1}$.
Thus, $\b_0=\b_1=\cdots=\b_{m-1}$. Finally, from  $f(ax_n^{m-1})=af(x_n^{m-1})$, one gets $\eta=\eta'$.
\end{proof}

\begin{corollary}\label{4.28}
Let $1\<l<n$, $i\in\mathbb Z$ and $\eta\in\Bbbk^{\times}$. Then
$${\rm End}_{H_{mn}(\xi)}(M_1(l,i,\eta))\cong\Bbbk.$$
\end{corollary}

\begin{proof}
It follows from the proof of Proposition \ref{4.27}.
\end{proof}

\begin{lemma}\label{4.29}
Let $M$ be an indecomposable $H_{mn}(\xi)$-module of $(t,t)$-type. Then either $M$ contains a submodule of $(1,1)$-type, or $M$ contains a submodule isomorphic to $M_1(l,i,\eta)$ for some $1\<l<n$, $i\in\mathbb Z$ and $\eta\in\Bbbk^{\times}$.
\end{lemma}

\begin{proof}
If $t=1$, the result is trivial. Now assume $t>1$.
Then by Proposition \ref{4.1}, $P=\oplus_{k=0}^{m-1}t_kP(l,i+kn)$ is an envelope of $M$ for some $1\<l\<n-1$, $i\in\mathbb Z$ and $t_k\in\mathbb N$ with $\sum_{k=0}^{m-1}t_k=t$. Hence we may assume $M\subseteq{\rm soc}^2P={\rm rad}P=\oplus_{k=0}^{m-1}t_k{\rm rad}P(l,i+kn)$. Then ${\rm soc}M={\rm soc}P=\oplus_{k=0}^{m-1}t_k{\rm soc}P(l,i+kn)$.

For any $1\<s\<t_k$, let $\{v_1^{k,s}, \cdots, v_{n-l}^{k,s}, u_1^{k,s}, \cdots, u_n^{k,s}\}$ be a standard basis of the $s^{th}$ summand ${\rm rad}P(l,i+kn)$ of $t_k{\rm rad}P(l,i+kn)$. Then ${\rm rad}P$ has a $\Bbbk$-basis
$$\{v_1^{k,s}, \cdots, v_{n-l}^{k,s}, u_1^{k,s}, \cdots, u_n^{k,s}|0\<k\<m-1, 1\<s\<t_k\}$$
and ${\rm soc}M={\rm soc}P={\rm span}\{u_1^{k,s}, \cdots, u_l^{k,s}|0\<k\<m-1,1\<s\<t_k\}$.
Let $V={\rm span}\{v_1^{k,s}, \cdots, v_{n-l}^{k,s}, u_{l+1}^{k,s}, \cdots, u_n^{k,s}|0\<k\<m-1,1\<s\<t_k\}$ and $U=M\cap V$. Then as vector spaces,
${\rm rad}P=V\oplus{\rm soc}P$ and $M=U\oplus{\rm soc}M$. For any $0\<j\<l-1$ and $0\<k\<m-1$, $({\rm rad}P)_{(i+j+kn)}={\rm span}\{u_{j+1}^{k,s}|1\<s\<t_k\}$.
For any $l\<j\<n-1$ and $0\<k\<m-1$, $({\rm rad}P)_{(i+j+kn)}={\rm span}\{v_{j+1-l}^{k+1,s'}, u_{j+1}^{k,s}|1\<s'\<t_{k+1},1\<s\<t_k\}$, where we regard $t_m=t_0$ and $v_{j+1-l}^{m,s'}=v_{j+1-l}^{0,s'}$.
Note that ${\rm soc}P=\oplus_{j=0}^{l-1}\oplus_{k=0}^{m-1}({\rm rad}P)_{(i+j+kn)}$
and $V=\oplus_{j=l}^{n-1}\oplus_{k=0}^{m-1}({\rm rad}P)_{(i+j+kn)}$.
Hence ${\rm soc}M=\oplus_{j=0}^{l-1}\oplus_{k=0}^{m-1}M_{(i+j+kn)}$
and $U=\oplus_{j=l}^{n-1}\oplus_{k=0}^{m-1}M_{(i+j+kn)}$. Moreover,
$M_{(i+j+kn)}=({\rm rad}P)_{(i+j+kn)}$ for all $0\<j\<l-1$ and $0\<k\<m-1$.
For any $0\<j\<n-1$, let $P_{[j]}=\oplus_{k=0}^{m-1}({\rm rad}P)_{(i+j+kn)}$ and $M_{[j]}=M\cap P_{[j]}$. Then ${\rm rad}P=\oplus_{j=0}^{n-1}P_{[j]}$, $M_{[j]}=\oplus_{k=0}^{m-1}M_{(i+j+kn)}$ and $M=\oplus_{j=0}^{n-1}M_{[j]}$.
Moreover, $M_{[j]}=P_{[j]}$ for all $0\<j\<l-1$, ${\rm soc}M=\oplus_{j=0}^{l-1}M_{[j]}$ and $U=\oplus_{j=l}^{n-1}M_{[j]}$.
By the structure of ${\rm rad}P$, the maps $M_{[j]}\ra M_{[j-1]}, x\mapsto dx$ and $M_{[j-1]}\ra M_{[j]}, x\mapsto ax$ are both bijective for any $0<j\<n-1$ with $j\neq l$.
Moreover, $dM_{[0]}=0$, $dM_{[l]}\subseteq M_{[l-1]}$, $aM_{[l-1]}=0$ and $aM_{[n-1]}\subseteq M_{[0]}$. It follows that ${\rm dim}M_{[j]}=t$ for all $0\<j\<n-1$.

If $M^d\cap M_{[l]}\neq 0$, then it follows from Lemma \ref{4.22}(2) and $M_{[l]}=\oplus_{k=0}^{m-1}M_{(i+l+kn)}$ that $M^d\cap M_{(i+l+kn)}\neq 0$
for some $0\<k\<m-1$. Since $M_{(i+l+kn)}\subseteq({\rm rad}P)_{(i+l+kn)}$, it follows from the action of $d$ on ${\rm rad}P$ that $M^d\cap M_{(i+l+kn)}\subseteq{\rm span}\{v_1^{k+1,s}|1\<s\<t_{k+1}\}$. Let $0\neq x\in M^d\cap M_{(i+l+kn)}$. Then
$x=\sum_{s=1}^{t_{k+1}}\b_sv_1^{k+1,s}$ for some $\b_1,\cdots, \b_{t_{k+1}}\in\Bbbk$. Then by the proof of Proposition \ref{4.12}, one can see that the submodule $\langle x\rangle$ is isomorphic to $\ol{T}_1(l,i+(k+1)n)$, and hence $M$ contains a submodule of $(1,1)$-type.

Now assume $M^d\cap M_{[l]}=0$. For any $0\<k\<m-1$, let $V_k^0={\rm span}\{v_1^{k+1, s}|1\<s\<t_{k+1}\}$ and $V_k^1={\rm span}\{u_{l+1}^{k, s}|1\<s\<t_k\}$. Then
$({\rm rad}P)_{(i+l+kn)}=V_k^0\oplus V_k^1$ and $M_{(i+l+kn)}\cap V_k^0=0$.
Note that $V_k^0=0$ if $t_{k+1}=0$.
Hence $V_k^0\oplus M_{(i+l+kn)}\subseteq V_k^0\oplus V_k^1$ and so ${\rm dim} M_{(i+l+kn)}\<{\rm dim}V_k^1=t_k$. However, $\sum_{k=0}^{m-1}{\rm dim} M_{(i+l+kn)}={\rm dim}M_{[l]}=t=\sum_{k=0}^{m-1}t_k$. It follows that ${\rm dim} M_{(i+l+kn)}={\rm dim}V_k^1=t_k$ and so $({\rm rad}P)_{(i+l+kn)}=V_k^0\oplus V_k^1=V_k^0\oplus M_{(i+l+kn)}$
for all $0\<k\<m-1$.
If $t_0, t_1, \cdots, t_{m-1}$ are not all nonzero, then we may assume $t_{k+1}=0$ but $t_k\neq 0$ for some $0\<k\<m-1$. In this case, $M_{(i+l+kn)}=V_k^1={\rm span}\{u_{l+1}^{k, s}|1\<s\<t_k\}$, and hence the submodule $\langle u_{l+1}^{k,1}\rangle$ of $M$ is isomorphic to $T_1(l,i+kn)$ by the proof of Proposition \ref{4.12}. Thus, $M$ contains a submodule of $(1,1)$-type. From now on, assume that $t_0, t_1, \cdots, t_{m-1}$ are all nonzero. Then for any $0\<k\<m-1$, there exists a basis $\{x_{k,s}|1\<s\<t_k\}$ of $M_{(i+l+kn)}$ such that $x_{k,s}-u_{l+1}^{k,s}\in V_k^0$ for all $1\<s\<t_k$. Firstly, suppose that $x_{k,1}-u_{l+1}^{k,1}, \cdots, x_{k,t_k}-u_{l+1}^{k,t_k}$ are linearly dependent over $\Bbbk$ for some $0\<k\<m-1$. If $t_k=1$ then $x_{k,1}-u_{l+1}^{k,1}=0$. In this case, $u_{l+1}^{k,1}=x_{k,1}\in M$, and $M$ contains a submodule $\langle u_{l+1}^{k,1}\rangle$ of $(1,1)$-type as above.
If $t_k>1$, then without losing generality, we may assume that $x_{k,t_k}-u_{l+1}^{k,t_k}=\sum_{j=1}^{t_k-1}\b_j(x_{k,j}-u_{l+1}^{k,j})$ for some $\b_j\in\Bbbk$. In this case, $u_{l+1}^{k,t_k}-\sum_{j=1}^{t_k-1}\b_ju_{l+1}^{k,j}=x_{k,t_k}-\sum_{j=1}^{t_k-1}\b_jx_{k,j}\in M$
and  $M$ contains a  submodule $\langle u_{l+1}^{k,t_k}-\sum_{j=1}^{t_k-1}\b_ju_{l+1}^{k,j}\rangle$ of $(1,1)$-type as above. Then suppose that $x_{k,1}-u_{l+1}^{k,1}, \cdots, x_{k,t_k}-u_{l+1}^{k,t_k}$ are linearly independent over $\Bbbk$ for all $0\<k\<m-1$. Then $t_{k+1}={\rm dim}V_k^0\>t_k$ for all $0\<k\<m-1$. Hence $t_0=t_m\>t_{m-1}\>\cdots\>t_1\>t_0$, and so $t_0=t_1=\cdots=t_{m-1}$. Thus, for any $0\<k\<m-1$, there is an invertible matrix $X_k\in M_{t_0}(\Bbbk)$ such that
$$(x_{k,1}-u_{l+1}^{k,1}, x_{k,2}-u_{l+1}^{k,2}, \cdots, x_{k,t_0}-u_{l+1}^{k,t_0})=(v_1^{k+1,1}, v_1^{k+1,2}, \cdots, v_1^{k+1,t_0})X_k.$$
Then $X=X_{m-1}\cdots X_1X_0$ is an invertible matrix in $M_{t_0}(\Bbbk)$. Since $\Bbbk$ is an algebraically closed field, there is a nonzero element $B=(\b_1, \cdots, \b_{t_0})^T\in\Bbbk^{t_0}$ and a nonzero scalar $\eta\in\Bbbk$ such that $XB=\eta B$, where $(\b_1, \cdots, \b_{t_0})^T$ denotes the transposition of $(\b_1, \cdots, \b_{t_0})$. For $0\<k\<m-1$, define $x_{l+1}^k\in M_{(i+l+kn)}$ by
$$\begin{array}{l}
x_{l+1}^0=(x_{0,1},\cdots,x_{0,t_0})B,\
x_{l+1}^1=(x_{1,1},\cdots,x_{1,t_0})X_0B,\
\cdots,\\
x_{l+1}^{m-1}=(x_{m-1,1},\cdots,x_{m-1,t_0})X_{m-2}\cdots X_1X_0B.
\end{array}$$
Then we have
$$\begin{array}{rl}
x_{l+1}^0=&(u_{l+1}^{0,1},\cdots,u_{l+1}^{0,t_0})B+(v_1^{1,1},\cdots,v_1^{1,t_0})X_0B,\\
x_{l+1}^1=&(u_{l+1}^{1,1},\cdots,u_{l+1}^{1,t_0})X_0B+(v_1^{2,1},\cdots,v_1^{2,t_0})X_1X_0B,\\
&\cdots\\
x_{l+1}^{m-1}=&(u_{l+1}^{m-1,1},\cdots,u_{l+1}^{m-1,t_0})X_{m-2}\cdots X_0B+(v_1^{m,1},\cdots,v_1^{m,t_0})X_{m-1}\cdots X_0B\\
=&(u_{l+1}^{m-1,1},\cdots,u_{l+1}^{m-1,t_0})X_{m-2}\cdots X_0B+\eta(v_1^{0,1},\cdots,v_1^{0,t_0})B.\\
\end{array}$$
Note that the map $P_{[n-1]}\ra P_{[l]}, x\mapsto d^{n-l-1}x$ is a bijection and its restriction gives rise to a bijection from $M_{[n-1]}$ onto $M_{[l]}$. Hence if $x\in P_{[n-1]}$ satisfies $d^{n-l-1}x\in M_{[l]}$, then $x\in M_{[n-1]}$.
Define $x_n^0, x_n^1, \cdots, x_n^{m-1}\in P_{[n-1]}$ by
$$\begin{array}{rl}
x_n^0=&(u_n^{0,1},\cdots,u_n^{0,t_0})B+(v_{n-l}^{1,1},\cdots,v_{n-l}^{1,t_0})X_0B,\\
x_n^1=&(u_n^{1,1},\cdots,u_n^{1,t_0})X_0B+(v_{n-l}^{2,1},\cdots,v_{n-l}^{2,t_0})X_1X_0B,\\
&\cdots\\
x_n^{m-1}=&(u_n^{m-1,1},\cdots,u_n^{m-1,t_0})X_{m-2}\cdots X_0B+\eta(v_{n-l}^{0,1},\cdots,v_{n-l}^{0,t_0})B.\\
\end{array}$$
Let $0\<k\<m-1$. It is easy to see that $d^{n-l-1}x_n^k=x_{l+1}^k$, and hence $x_n^k\in M_{[n-1]}$. Let $x_j^k=d^{n-j}x_n^k$ for $1\<j<n$. Then $x_j^k\in M_{(i+j-1+kn)}$ for all $1\<j\<n$. Moreover, $dx_j^k=x_{j-1}^k$ for all $1<j\<n$, and $dx_1^k=0$. Furthermore, we have
$$\begin{array}{l}
x_j^0=(u_j^{0,1},\cdots,u_j^{0,t_0})B,\
x_j^1=(u_j^{1,1},\cdots,u_j^{1,t_0})X_0B,\
\cdots,\\
x_j^{m-1}=(u_j^{m-1,1},\cdots,u_j^{m-1,t_0})X_{m-2}\cdots X_0B\\
\end{array}$$
for all $1\<j\<l$, and
$$\begin{array}{rl}
x_j^0=&(u_j^{0,1},\cdots,u_j^{0,t_0})B+(v_{j-l}^{1,1},\cdots,v_{j-l}^{1,t_0})X_0B,\\
x_j^1=&(u_j^{1,1},\cdots,u_j^{1,t_0})X_0B+(v_{j-l}^{2,1},\cdots,v_{j-l}^{2,t_0})X_1X_0B,\\
&\cdots\\
x_j^{m-1}=&(u_j^{m-1,1},\cdots,u_j^{m-1,t_0})X_{m-2}\cdots X_0B+\eta(v_{j-l}^{0,1},\cdots,v_{j-l}^{0,t_0})B.\\
\end{array}$$
for all $l+1<j<n$. Clearly, the set $\{x_j^k|1\<j\<n, 0\<k\<m-1\}$ is linearly independent.
By Lemma \ref{3.6}(4,7), a straightforward verification shows that
\begin{eqnarray*}
&ax_j^k=\left\{\begin{array}{ll}
\a_j(l,i)x_{j+1}^k,& 1\<j<n, 0\<k\<m-1,\\
\a_n(l,i)x_1^k+x_1^{k+1},& j=n, 0\<k<m-1,\\
\a_n(l,i)x_1^{m-1}+\eta x_1^0,& j=n, k=m-1,\\
\end{array}\right.\\
&bx_j^k=q^{i+l-j}x_j^k,\ 1\<j\<n, 0\<k\<m-1.
\end{eqnarray*}
It follows that $\{x_j^k|1\<j\<n, 0\<k\<m-1\}$ is a basis of the submodule $\langle x_n^0, x_n^1, \cdots, x_n^{m-1}\rangle$ of $M$ and $\langle x_n^0, x_n^1, \cdots, x_n^{m-1}\rangle\cong M_1(l,i,\eta)$.
\end{proof}

\begin{corollary}\label{4.30}
Let $M$ be an indecomposable $H_{mn}(\xi)$-module of $(t,t)$-type. If $m\nmid t$, then $M$ contains a submodule of $(1,1)$-type.
\end{corollary}

\begin{proof}
It follows from the proof of Lemma \ref{4.29}.
\end{proof}

\begin{corollary}\label{4.31}
Let $M$ be a nonzero submodule of $M_1(l,i,\eta)$, where $1\<l\<n-1$, $i\in\mathbb Z$ and $\eta\in\Bbbk^{\times}$. If $l(M/{\rm soc}M)=l({\rm soc}M)$, then $M=M_1(l,i,\eta)$.
\end{corollary}

\begin{proof}
Let $s:=l(M/{\rm soc}M)=l({\rm soc}M)$. If $s<m$, then it follows from Lemma \ref{4.8} that
$M$ contains a direct summand $N$ of $(t,t)$-type with $1\<t\<s$. By Corollary \ref{4.30}, $N$ contains a submodule of $(1,1)$-type, and so does $M_1(l,i,\eta)$, which contradicts Lemma \ref{4.23}(3). Hence $s=m$, and so $M=M_1(l,i,\eta)$.
\end{proof}

Throughout the following, unless otherwise stated, let $1\<l\<n-1$, $i\in\mathbb Z$ and $\eta\in\Bbbk^{\times}$.

By Corollary \ref{4.26}, $\O^2M_1(l,i,\eta)\cong M_1(l,i,\eta)$. Hence there exists an AR-sequence
$$0\ra M_1(l,i,\eta)\xrightarrow{f_1}M_2(l,i,\eta)\xrightarrow{g_1}M_1(l,i,\eta)\ra 0$$
for a unique module $M_2(l,i,\eta)$ up to isomorphism.

\begin{lemma}\label{4.32}
With the notations above, $M_2(l,i,\eta)$ is an indecomposable module of $(2m,2m)$-type,
$\O^2M_2(l,i,\eta)\cong M_2(l,i,\eta)$, ${\rm soc}M_2(l,i,\eta)\cong\oplus_{k=0}^{m-1}2M(l,i+kn)$ and
     $M_2(l,i,\eta)/{\rm soc}M_2(l,i,\eta)\cong\oplus_{k=0}^{m-1}2M(n-l,i+l+kn)$.
\end{lemma}

\begin{proof}
Suppose on the contrary that $M_2(l,i,\eta)$ has a nontrivial indecomposable direct summand $N$. Then $g_1|_N: N\ra M_1(l,i,\eta)$ is an irreducible
morphism. By Lemma \ref{4.14}, $N$ is of $(t,t)$-type and $1\<t<2m$. Clearly, $t\neq m$, and hence $N$ contains a submodule $K$ of $(1,1)$-type  by Corollary \ref{4.30}.
By Lemma \ref{4.9}, $K\subseteq{\rm Ker}g_1$ or $K\cap{\rm Ker}g_1=0$.
Since ${\rm Ker}g_1={\rm Im}f_1\cong M_1(l,i,\eta)$, $K\nsubseteq{\rm Ker}g_1$ by Lemma \ref{4.23}(3). Hence
$K\cap{\rm Ker}g_1=0$. This implies that $K\cong g_1(K)\subseteq M_1(l,i,\eta)$, which contradicts Lemma \ref{4.23}(3). This shows that $M_2(l,i,\eta)$ is indecomposable. Obviously, $M_2(l,i,\eta)$ is of $(2m,2m)$-type. Since $\O^2M_1(l,i,\eta)\cong M_1(l,i,\eta)$, one gets  $\O^2M_2(l,i,\eta)\cong M_2(l,i,\eta)$. Clearly, ${\rm soc}M_2(l,i,\eta)\cong\oplus_{k=0}^{m-1}2M(l,i+kn)$ and  $M_2(l,i,\eta)/{\rm soc}M_2(l,i,\eta)\cong\oplus_{k=0}^{m-1}2M(n-l,i+l+kn)$.
\end{proof}

\begin{lemma}\label{4.33}
Let $N$ be a submodule of $(t,t)$-type of $M_2(l,i,\eta)$. Then either $N={\rm Im}f_1$ or $N=M_2(l,i,\eta)$.
Moreover, $M_2(l,i,\eta)/{\rm Im}f_1\cong M_1(l,i,\eta)$.
\end{lemma}

\begin{proof}
Let $L:=N\cap{\rm Im}f_1$. Then $l(L/{\rm soc}L)=l({\rm soc}L)$ by Lemma \ref{4.9}.
If $L=0$, then $N\cong g_1(N)\subseteq M_1(l,i,\eta)$ by ${\rm Im}f_1={\rm Ker}g_1$. In this case, by Corollary \ref{4.31}, $g_1(N)=M_1(l,i,\eta)$, and so $M_2(l,i,\eta)=N\oplus{\rm Im}f_1$, a contradiction. Hence $L\neq 0$. Note that $L\subseteq{\rm Im}f_1\cong M_1(l,i,\eta)$.  Again by Corollary \ref{4.31}, $L={\rm Im}f_1$, and hence ${\rm Im}f_1\subseteq N$. This implies $m\<t\<2m$. If $m<t<2m$, then $N$ contains a submodule $K$ of $(1,1)$-type by Corollary \ref{4.30}. By Lemma \ref{4.9}, one knows that either $K\subseteq{\rm Im}f_1$ or $K\cap{\rm Im}f_1=0$.
Since ${\rm Im}f_1\cong M_1(l,i,\eta)$, $K\nsubseteq{\rm Im}f_1$ by Lemma \ref{4.23}(3). Hence $K\cap{\rm Im}f_1=0$, which implies $K\cong g_1(K)\subseteq M_1(l,i,\eta)$, a contradiction. Hence $t=m$ or $t=2m$. Thus, either $N={\rm Im}f_1$ or $N=M_2(l,i,\eta)$.
Moreover, $M_2(l,i,\eta)/{\rm Im}f_1=M_2(l,i,\eta)/{\rm Ker}g_1\cong M_1(l,i,\eta)$.
\end{proof}

\begin{corollary}\label{4.34}
For any $1\<l'\<n-1$ and $i'\in\mathbb Z$, $M_2(l,i,\eta)\ncong T_{2m}(l',i')$ and $M_2(l,i,\eta)\ncong\ol{T}_{2m}(l',i')$.
\end{corollary}

\begin{proof}
By Lemma \ref{4.33}, $M_2(l,i,\eta)$ does not contain any submodule of $(1,1)$-type.
However, both $T_{2m}(l',i')$ and $\ol{T}_{2m}(l',i')$ contain a submodule of $(1,1)$-type by Theorems \ref{4.18}(2) and \ref{4.19}(2). Thus, the corollary follows.
\end{proof}

Since $\O^2M_2(l,i,\eta)\cong M_2(l,i,\eta)$, by \cite[Proposiotns V.1.12 and 5.3]{ARS}, there is a unique $H_{mn}(\xi)$-module $M_3(l,i,\eta)$ up to isomorphism which fits an AR-sequence
\begin{eqnarray*}
0\rightarrow M_2(l,i,\eta)\xrightarrow{\left(
	\begin{matrix}
	g_1 \\
	f_2 \\
	\end{matrix}
	\right)
} M_1(l,i,\eta)\oplus M_3(l,i,\eta)\xrightarrow{(f'_1, g_2)} M_2(l,i,\eta)\rightarrow 0.
\end{eqnarray*}

\begin{lemma}\label{4.35}
Retain the above notations.
\begin{enumerate}
\item[(1)] $M_3(l,i,\eta)$ is an indecomposable module of $(3m,3m)$-type,
${\rm soc}M_3(l,i,\eta)\cong\oplus_{k=0}^{m-1}3M(l,i+kn)$ and
$M_3(l,i,\eta)/{\rm soc}M_3(l,i,\eta)\cong\oplus_{k=0}^{m-1}3M(n-l,i+l+kn)$.
\item[(2)] If $N$ is a submodule of $(t,t)$-type of $M_3(l,i,\eta)$ with $1\<t<3m$, then $N={\rm Im}(f_2f_1)$ or ${\rm Im}f_2$. Moreover, $M_3(l,i,\eta)/{\rm Im}(f_2f_1)\cong M_2(l,i,\eta)$ and $M_3(l,i,\eta)/{\rm Im}f_2\cong M_1(l,i,\eta)$.
\item[(3)] $\O^2M_3(l,i,\eta)\cong M_3(l,i,\eta)$.
\item[(4)] $M_3(l,i,\eta)\ncong T_{3m}(l',i')$ and $M_3(l,i,\eta)\ncong\ol{T}_{3m}(l',i')$ for any $1\<l'\<n-1$ and $i\in\mathbb Z$.
\end{enumerate}
\end{lemma}
\begin{proof}
(1) Suppose on the contrary that $M_3(l,i,\eta)$ is decomposable. By Lemma \ref{4.14},  $M_3(l,i,\eta)$ has a nontrivial indecomposable direct summand $N$ of $(s,s)$-type with $1\<s<3m$. Then $g_2|_N: N\ra M_2(l,i,\eta)$ is an irreducible morphism. Hence $g_2|_N$ is injective or surjective, but $g_2|_N$ is not bijective.  If $g_2|_N$ is surjective, then $2m<s<3m$, and hence $N$ contains a submodule of $(1,1)$-type by Corollary \ref{4.30}. Thus, by Proposition \ref{4.21},
$N\cong T_s(l',i')$ or $N\cong\ol{T}_s(l',i')$ for some $1\<l'\<n-1$ and $i'\in\mathbb Z$. This implies that there is an irreducible morphism from $T_s(l',i')$ to $M_2(l,i,\eta)$ or from $\ol{T}_s(l',i')$ to $M_2(l,i,\eta)$, which contradicts
Theorem \ref{4.18}(4) or Theorem \ref{4.19}(4) by Corollary \ref{4.34}. Hence $g_2|_N$ is injective. So $s<2m$ and $N\cong g_2(N)\subsetneqq M_2(l,i,\eta)$. Then by Lemma \ref{4.33}, $s=m$ and $g_2(N)={\rm Im}f_1\cong M_1(l,i,\eta)$.
Thus, $M_3(l,i,\eta)$ has to be decomposed into a direct sum of three indecomposable submodules of $(m,m)$-type, and
hence ${\rm Im}(f_1',g_2)={\rm Im}f_1$, a contradiction. This shows that $M_3(l,i,\eta)$ is an indecomposable module of $(3m,3m)$-type. Clearly, ${\rm soc}M_3(l,i,\eta)\cong\oplus_{k=0}^{m-1}3M(l,i+kn)$ and $M_3(l,i,\eta)/{\rm soc}M_3(l,i,\eta)\cong\oplus_{k=0}^{m-1}3M(n-l,i+l+kn)$.

(2) Let $N$ be a submodule of $(t,t)$-type of $M_3(l,i,\eta)$ with $1\<t<3m$. If $m\nmid t$, then
by Corollary \ref{4.30}, $N$ contains a submodule of $(1,1)$-type, and so does $M_3(l,i,\eta)$.
By Proposition \ref{4.21}, $M_3(l,i,\eta)\cong T_{3m}(l',i')$ or $M_3(l,i,\eta)\cong\ol{T}_{3m}(l',i')$ for some $1\<l'\<n-1$ and $i'\in\mathbb Z$. This contradicts Theorem \ref{4.18}(4) or Theorem \ref{4.19}(4) as stated above, since $g_2$ is an irreducible morphism. Hence $m|t$.

Case 1: $t=m$. Since $f_2: M_2(l,i,\eta)\rightarrow M_3(l,i,\eta)$ is an irreducible morphism, $f_2$ is injective, and so ${\rm Im}f_2$ is a submodule of $(2m,2m)$-type of $M_3(l,i,\eta)$. If $N\cap {\rm Im}f_2=0$, then $M_3(l,i,\eta)=N\oplus {\rm Im}f_2$, which is impossible. Hence $L:=N\cap{\rm Im}f_2\neq 0$. By Lemma \ref{4.9}, $l(L/{\rm soc}L)=l({\rm soc}L)\<m$. By Lemma \ref{4.8}, $L$ contains a submodule $K$ of $(s,s)$-type with $s\<m$. Then by the previous paragraph, $s=m$. This forces $N=L\subsetneqq{\rm Im}f_2\cong M_2(l,i,\eta)$, and so $N={\rm Im}(f_2f_1)\cong M_1(l,i,\eta)$ by Lemma \ref{4.33}. Since $g_2$ is an irreducible morphism, it is surjective, and hence $l({\rm Ker}g_2/{\rm soc}({\rm Ker}g_2))=l({\rm soc}({\rm Ker}g_2))=m$. If ${\rm Ker}g_2$ is decomposable, then by Lemma \ref{4.8}, ${\rm Ker}g_2$ has a direct summand of $(s,s)$-type with $1\<s<m$, a contradiction. Hence ${\rm Ker}g_2$ is an indecoposable module of $(m,m)$-type. Thus, ${\rm Ker}g_2={\rm Im}(f_2f_1)$, and so $M_3(l,i,\eta)/{\rm Im}(f_2f_1)=M_3(l,i,\eta)/{\rm Ker}g_2\cong M_2(l,i,\eta)$.

Case 2: $t=2m$. In this case, $K:=N\cap {\rm Ker}g_2\neq 0$ since $M_3(l,i,\eta)$ is indecomposable, and $l(K/{\rm soc}K)=l({\rm soc}K)$ by Lemma \ref{4.9}.
Then by Corollary \ref{4.31} and $K\subseteq{\rm Ker}g_2={\rm Im}(f_2f_1)\cong M_1(l,i,\eta)$, one gets $K={\rm Ker}g_2$, and so ${\rm Ker}g_2\subseteq N$. In particular, ${\rm Ker}g_2\subseteq{\rm Im}f_2$. Hence ${\rm Ker}g_2\subseteq N\cap{\rm Im}f_2$.
If ${\rm Ker}g_2=N\cap{\rm Im}f_2$, then $l(N+{\rm Im}f_2)=l(N)+l({\rm Im}f_2)-l({\rm Ker}g_2)=6m$, and so $N+{\rm Im}f_2=M_3(l,i,\eta)$, which implies $M_2(l,i,\eta)=g_2(N+{\rm Im}f_2)=g_2(N)\oplus g_2({\rm Im}f_2)$, a contradiction. Hence ${\rm Ker}g_2\neq N\cap {\rm Im}f_2$. Let $L_1:=N\cap {\rm Im}f_2$. Then $l(L_1/{\rm soc}L_1)=l({\rm soc}L_1)$ by Lemma \ref{4.9}. If $L_1\neq {\rm Im}f_2$, $m<l({\rm soc}L_1)<2m$. In this case, whether $L_1$ is indecomposable or not, $L_1$ contains an indecomposable module of $(s_1,s_1)$ with $m\nmid s_1$, a contradiction.
Hence $L_1={\rm Im}f_2$, and so $N={\rm Im}f_2$.
Now by Lemma \ref{4.33}, ${\rm Im}f_2/{\rm Im}(f_2f_1)\cong M_1(l,i,\eta)$ and hence $M_3(l,i,\eta)/{\rm Im}f_2\cong(M_3(l,i,\eta)/{\rm Im}(f_2f_1))/({\rm Im}f_2/{\rm Im}(f_2f_1))\cong M_2(l,i,\eta)/{\rm Im}f_1\cong M_1(l,i,\eta)$.

 (3) Since $\O^2M_1(l,i,\eta)\cong M_1(l,i,\eta)$ and $\O^2M_2(l,i,\eta)\cong M_2(l,i,\eta)$, an argument similar to $T_2(l,i)$ shows that $\O^2M_3(l,i,\eta)\cong M_3(l,i,\eta)$.

(4) By  (2) and Theorems \ref{4.18}(2) and \ref{4.19}(2), $M_3(l,i,\eta)\ncong T_{3m}(l',i')$ and $M_3(l,i,\eta)\ncong\ol{T}_{3m}(l',i')$ for any $1\<l'\<n-1$ and $i\in\mathbb Z$.
\end{proof}

Now suppose $t\>3$ and we have got indecomposable modules
$$M_1(l,i,\eta), M_2(l,i,\eta),\cdots,  M_t(l,i,\eta)$$
and AR-sequences
\begin{eqnarray*}
&0\rightarrow M_1(l,i,\eta)\xrightarrow{f_1}M_2(l,i,\eta)\xrightarrow{g_1} M_1(l,i,\eta)\rightarrow 0,\\
&0\rightarrow M_{s-1}(l,i,\eta)\xrightarrow{\left(\begin{matrix}
	g_{s-2} \\
	f_{s-1} \\
	\end{matrix}\right)}
M_{s-2}(l,i,\eta)\oplus M_s(l,i,\eta)\xrightarrow{(f'_{s-2},g_{s-1})}
M_{s-1}(l,i,\eta)\rightarrow 0
\end{eqnarray*}
for all $3\<s\<t$. They satisfy the following:
\begin{enumerate}
\item[(1)] $M_s(l,i,\eta)$ is of $(sm,sm)$-type with ${\rm soc}M_s(l,i,\eta)\cong\oplus_{k=0}^{m-1}sM(l,i+kn)$ and\\
$M_s(l,i,\eta)/{\rm soc}M_s(l,i,\eta)\cong \oplus_{k=0}^{m-1}sM(n-l,i+l+kn)$;
\item[(2)] If $N$ is a submodule of $(s',s')$-type of $M_s(l,i,\eta)$ with $1\<s'<sm$,
then $s'=jm$ and $N={\rm Im}(f_{s-1}\cdots f_j)$ for some $1\<j\<s-1$. Moreover,
$M_s(l,i,\eta)/{\rm Im}(f_{s-1}\cdots f_j)\cong M_{s-j}(l,i,\eta)$ for all $1\<j\<s-1$;
\item[(3)] $\Omega^2M_s(l,i,\eta)\cong M_s(l,i,\eta)$;
\item[(4)]  $M_s(l,i,\eta)\ncong T_{sm}(l',i')$ and $M_s(l,i,\eta)\ncong\ol{T}_{sm}(l',i')$ for any $1\<l'\<n-1$ and $i'\in\mathbb Z$,
\end{enumerate}
for all $1\<s\<t$. Again by \cite[Proposiotns V.1.12 and 5.3]{ARS}, there is a unique module $M_{t+1}(l,i,\eta)$ up to isomorphism such that there is an AR-sequence
\begin{eqnarray*}
0\rightarrow M_t(l,i,\eta)\xrightarrow{\left(
	\begin{matrix}
	g_{t-1} \\
	f_t \\
	\end{matrix}
	\right)
} M_{t-1}(l,i,\eta)\oplus M_{t+1}(l,i,\eta)\xrightarrow{(f'_{t-1}, g_t)} M_t(l,i,\eta)\rightarrow 0.
\end{eqnarray*}

\begin{lemma}\label{4.36}
Retain the above notations.
\begin{enumerate}
\item[(1)] $M_{t+1}(l,i,\eta)$ is an indecomposable module of $((t+1)m,(t+1)m)$-type with ${\rm soc}M_{t+1}(l,i,\eta)\cong \oplus_{k=0}^{m-1}(t+1)M(l,i+kn)$, $M_{t+1}(l,i,\eta)/{\rm soc}M_{t+1}(l,i,\eta)\cong \oplus_{k=0}^{m-1}(t+1)M(n-l,i+l+kn)$.
\item[(2)] If $N$ is an indecomposable submodule of $(s,s)$-type of $M_{t+1}(l,i,\eta)$ with $1\<s<(t+1)m$,
then $s=jm$ and $N={\rm Im}(f_t\cdots f_j)$ for some $1\<j\<t$. Moreover, $M_{t+1}(l,i,\eta)/{\rm Im}(f_t\cdots f_j)\cong M_{t+1-j}(l,i,\eta)$ for all $1\<j\<t$.
\item[(3)] $\O^2M_{t+1}(l,i,\eta)\cong M_{t+1}(l,i,\eta)$.
\item[(4)] $M_{t+1}(l,i,\eta)\ncong T_{(t+1)m}(l',i')$ and $M_{t+1}(l,i,\eta)\ncong\ol{T}_{(t+1)m}(l',i')$ for any $1\<l'\<n-1$ and $i'\in\mathbb Z$.
\end{enumerate}
\end{lemma}
\begin{proof}
(1) If $M_{t+1}(l,i,\eta)$ is decomposable, then by Lemma \ref{4.14},  $M_{t+1}(l,i,\eta)$ has a nontrivial direct summand $N$ of $(s,s)$-type with $1\<s<(t+1)m$, and $g_t|_N: N\ra M_t(l,i,\eta)$ is an irreducible morphism. Hence $g_t|_N$ is a monomorphism or an epimorphism, but $g_t|_N$ is not an isomorphism.  Using the induction hypothesis (4), an argument similar to Lemma \ref{4.35} shows that  $g_t|_N$ is not an epimorphism. Hence $g_t|_N$ is a monomorphism. So $s<tm$ and $N\cong g_t(N)\subset M_t(l,i,\eta)$. Then by the induction hypothesis, $g_t(N)={\rm Im}(f_{t-1}\cdots f_j)\cong M_j(l,i,\eta)$ for some $1\<j\<t-1$, and so $g_t(N)\subseteq{\rm Im}f_{t-1}$
Thus, if $M_{t+1}(l,i,\eta)$ is decomposable, then $g_t(M_{t+1}(l,i,\eta))\subseteq{\rm Im}f_{t-1}$, and
hence ${\rm Im}(f'_{t-1},g_t)={\rm Im}f_{t-1}$, a contradiction. This shows that $M_{t+1}(l,i,\eta)$ is an indecomposable module of $((t+1)m,(t+1)m)$-type. Clearly, ${\rm soc}M_{t+1}(l,i,\eta)\cong\oplus_{k=0}^{m-1}(t+1)M(l,i+kn)$ and $M_{t+1}(l,i,\eta)/{\rm soc}M_{t+1}(l,i,\eta)\cong\oplus_{k=0}^{m-1}(t+1)M(n-l,i+l+kn)$.

(2) Let $N$ be a submodule of $(s,s)$-type of $M_{t+1}(l,i,\eta)$ with $1\<s<(t+1)m$. Then an argument similar to Lemma \ref{4.35} shows that $m|s$.

Let $N$ be a submodule of $(m,m)$-type of $M_{t+1}(l,i,\eta)$. Then similarly to the proof (Case1) of Lemma \ref{4.35}, one can show that $N={\rm Im}(f_t\cdots f_1)\cong M_1(l,i,\eta)$, ${\rm Ker}g_t={\rm Im}(f_t\cdots f_1)$ and $M_{t+1}(l,i,\eta)/{\rm Im}(f_t\cdots f_1)=M_{t+1}(l,i,\eta)/{\rm Ker}g_t\cong M_t(l,i,\eta)$.

Now let $1<j\<t$, and let $N$ be a submodule of $(jm,jm)$-type of $M_{t+1}(l,i,\eta)$. If $N\cap {\rm Ker}g_t=0$, then $j<t$, otherwise $M_{t+1}(l,i,\eta)=N\oplus {\rm Ker}g_t$, a contradiction. Hence $(N\oplus {\rm Ker}g_t)/{\rm Ker}g_t$ is a submodule of $(jm,jm)$-type of $M_{t+1}(l,i,\eta)/{\rm Ker}g_t$. Since $M_{t+1}(l,i,\eta)/{\rm Ker}g_t\cong M_t(l,i,\eta)$,  $M_{t+1}(l,i)/{\rm Ker}g_t$ contains a unique submodule of $(jm,jm)$-type. Since all $f_j$ are injective, ${\rm Im}(f_t\cdots f_{j+1})\cong M_{j+1}(l,i,\eta)$. Hence ${\rm Im}(f_t\cdots f_1)={\rm Ker}g_t$ is the unique submodule of $(m,m)$-type of ${\rm Im}(f_t\cdots f_{j+1})$ by the induction hypothesis and $j+1\<t$, and so
${\rm Im}(f_t\cdots f_{j+1})/{\rm Im}(f_t\cdots f_1)\cong M_{j}(l,i,\eta)$. Thus, ${\rm Im}(f_t\cdots f_{j+1})/{\rm Ker}g_t=(N\oplus{\rm Ker}g_t)/{\rm Ker}g_t$, and so ${\rm Im}(f_t\cdots f_{j+1})=N\oplus{\rm Ker}g_t$, a contradiction. This shows that $L:=N\cap{\rm Ker}g_t\neq 0$. Then $l(L/{\rm soc}L)=l({\rm soc}L)$ by Lemma \ref{4.9}. By Lemma \ref{4.8} and the result shown above, one can check $L={\rm Ker}g_t$, and hence ${\rm Ker}g_t\subseteq N$.
Thus, $N/{\rm Ker}g_t\subseteq M_{t+1}(l,i,\eta)/{\rm Ker}g_t\cong M_t(l,i,\eta)$. Then by Corollary \ref{4.4}, $l({\rm soc}(N/{\rm Ker}g_t))=l((N/{\rm Ker}g_t)/{\rm soc}(N/{\rm Ker}g_t))=(j-1)m$. By Lemma \ref{4.8}, any indecomposable direct summand of $N/{\rm Ker}g_t$ is of $(sm,sm)$-type. However, it follows by the induction hypothesis that there is a unique submodule of $(sm,sm)$-type in $M_{t+1}(l,i,\eta)/{\rm Ker}g_t$ for each $1\<s\<t$, and that the submodule of $(sm,sm)$-type is contained in that of $((s+1)m,(s+1)m)$-type for all $1\<s<t$. Thus, $N/{\rm Ker}g_t$ has to be an indecomposable module of $((j-1)m,(j-1)m)$-type. Then an argument as above shows that $N={\rm Im}(f_t\cdots f_j)$. Moreover, we have
$$\begin{array}{rl}
M_{t+1}(l,i,\eta)/{\rm Im}(f_t\cdots f_j)&\cong(M_{t+1}(l,i,\eta)/{\rm Ker}g_t)/({\rm Im}(f_t\cdots f_j)/{\rm Ker}g_t)\\
&\cong M_t(l,i,\eta)/{\rm Im}(f_{t-1}\cdots f_{j-1})\cong M_{t+1-j}(l,i,\eta).\\
\end{array}$$

(3) By the induction hypothesis, $\O^2M_{t-1}(l,i,\eta)\cong M_{t-1}(l,i,\eta)$ and $\O^2M_t(l,i,\eta)\cong M_t(l,i,\eta)$. Hence an argument similar to $T_2(l,i)$ shows that $\O^2M_{t+1}(l,i,\eta)\cong M_{t+1}(l,i,\eta)$.

(4) By (2) and Theorems \ref{4.18}(2) and \ref{4.19}(2), $M_{t+1}(l,i,\eta)\ncong T_{(t+1)m}(l',i')$ and $M_{t+1}(l,i,\eta)\ncong\ol{T}_{(t+1)m}(l',i')$ for any $1\<l'\<n-1$ and $i'\in\mathbb Z$.
\end{proof}

Summarizing the discussion above, one gets the following theorem.

\begin{theorem}\label{4.37}
For any $\eta\in\Bbbk^{\times}$ and $l, i, t\in \mathbb{Z}$ with $1\<l\<n-1$ and $t\>1$, there is an indecomposable $H_{mn}(\xi)$-module $M_t(l,i,\eta)$ of $(tm,tm)$-type. Moreover,
\begin{enumerate}
\item[(1)] $\O^2M_t(l,i,\eta)\cong M_t(l,i,\eta)$,  ${\rm soc}M_t(l,i,\eta)\cong\oplus_{k=0}^{m-1}tM(l,i+kn)$ and\\ $M_t(l,i,\eta)/{\rm soc}M_t(l,i,\eta)\cong\oplus_{k=0}^{m-1}tM(n-l,i+l+kn)$.
\item[(2)] If $M_t(l,i,\eta)$ contains a submodule of $(s,s)$-type, then $m|s$. Moreover, for any $1\<j<t$, $M_t(l,i,\eta)$ contains a unique submodule of $(jm,jm)$-type, which is isomorphic to $M_j(l,i,\eta)$ and the quotient module of $M_t(l,i,\eta)$ modulo the submodule of $(jm,jm)$-type is isomorphic to $M_{t-j}(l,i,\eta)$.
\item[(3)] For any $1\<j<t$, the unique submodule of $(jm,jm)$-type of $M_t(l,i,\eta)$ is contained in that of $((j+1)m,(j+1)m)$-type.
\item[(4)] $M_t(l,i,\eta)\ncong T_{tm}(l',i')$ and $M_t(l,i,\eta)\ncong\ol{T}_{tm}(l',i')$ for any $1\<l'\<n-1$ and $i'\in\mathbb Z$.
\item[(5)] There are AR-sequences
\begin{eqnarray*}
&0\ra M_1(l,i,\eta)\xrightarrow{f_1}
M_2(l,i,\eta)\xrightarrow{g_1}
M_1(l,i,\eta)\ra 0,\\
&0\ra M_t(l,i,\eta)\xrightarrow{\left(\begin{matrix}
	g_{t-1} \\
	f_t \\
	\end{matrix}\right)}
M_{t-1}(l,i,\eta)\oplus M_{t+1}(l,i,\eta)\xrightarrow{(f'_{t-1}, g_t)}
M_t(l,i,\eta)\ra 0\ (t\>2).
\end{eqnarray*}
\end{enumerate}
\end{theorem}

\begin{corollary}\label{4.38}
Let $\eta, \eta'\in\Bbbk^{\times}$ and $l, l', i, i', t, t'\in\mathbb Z$ with $1\<l, l'\<n-1$ and $t, t'\>1$. Then
$M_t(l,i,\eta)\cong M_{t'}(l', i', \eta') \Leftrightarrow t=t', l=l', \eta=\eta' \text{ and }i\equiv i'\ ({\rm mod}\ n).$
\end{corollary}

\begin{proof}
It follows from Proposition \ref{4.27}, Theorem \ref{4.37}(2,5), $l(M_t(l,i,\eta))=2tm$ and $l(M_{t'}(l', i', \eta'))=2t'm$.
\end{proof}

\begin{proposition}\label{4.39}
Let $M$ be an indecomposable $H_{mn}(\xi)$-module of $(t,t)$-type. If $M$ contains a submodule isomorphic to $M_1(l,i,\eta)$ for some $1\<l\<n-1$, $i\in\mathbb Z$ and $\eta\in\Bbbk^{\times}$, then $t=sm$ and $M\cong M_s(l,i, \eta)$ for some $s\>1$.
\end{proposition}

\begin{proof}
Assume that $M$ contains a submodule isomorphic to $M_1(l,i,\eta)$ for some $1\<l\<n-1$, $i\in\mathbb Z$ and $\eta\in\Bbbk^{\times}$. Then $t\>m$. By Proposition \ref{4.21}, Lemma \ref{4.23}(4), and Theorems \ref{4.18}(2,3) and \ref{4.19}(2,3), one can check that $M$ does not contain any submodule of $(1,1)$-type. Hence $m|t$ by Corollary \ref{4.30},
and so $t=sm$ for some $s\>1$. Then an argument similar to \cite[Theorem 4.16]{Ch3} shows that $M\cong M_t(l,i,\eta)$.
\end{proof}

Summarizing the discussions in the last section and this section, one gets the classification of finite dimensional indecomposable $H_{mn}(\xi)$-modules as follows.

\begin{corollary}\label{4.40}
A representative set of isomorphism classes of finite dimensional indecomposable $H_{mn}(\xi)$-modules is given by
$$\left\{\begin{array}{c}
M(l,i), Z(\l_{i'j'}), P(l,i), \Omega^{\pm s}M(l,i),\\
T_s(l,i), \ol{T}_s(l,i), M_s(l,j,\eta)\\
\end{array}\left|\begin{array}{c}
1\<l\<n-1, i\in\mathbb{Z}_{mn}, s\>1,\\
j\in\mathbb{Z}_n, (i',j')\in I_0, \eta\in\Bbbk^{\times}\\
\end{array}\right.\right\}.$$
\end{corollary}

By Corollary \ref{4.40}, we also have the following corollary.

\begin{corollary}
$H_{mn}(\xi)$ is of tame representation type.
\end{corollary}

\centerline{ACKNOWLEDGMENTS}

This work is supported by NNSF of China (Nos. 12201545, 12071412).\\

\end{document}